\documentclass[10pt]{article}
\usepackage{epsfig}
\usepackage{amssymb,amsmath,amsthm,url}
\usepackage{multirow}
\usepackage{booktabs}
\usepackage{natbib}
\usepackage{hyperref}
\usepackage{setspace}
\usepackage{dsfont}

\usepackage{amsfonts,times,eucal}
\usepackage{graphicx,ifthen}
\usepackage{wrapfig}
\usepackage{latexsym}
\usepackage{color}
\usepackage{xcolor}

\usepackage{tikz}

\usetikzlibrary{arrows}

\definecolor{blueCol}{rgb}{0.6,0.2,0}
\definecolor{redCol}{rgb}{0,0.2,0.8}
\definecolor{greenCol}{rgb}{0,0.8,0.2}

\usepackage{mathrsfs}
\usepackage{bm}
\usepackage{bbm}
\usepackage{srctex}
\usepackage{enumitem}
\usepackage{stackengine}
\stackMath

\makeatletter
\def\namedlabel#1#2{\begingroup
    #2%
    \def\@currentlabel{#2}%
    \phantomsection\label{#1}\endgroup
}
\makeatother

\usepackage{mathtools}
\DeclarePairedDelimiter{\floor}{\lfloor}{\rfloor}
\DeclarePairedDelimiter\ceil{ \lceil}{ \rceil}
\theoremstyle{plain}
\newtheorem{theorem}{Theorem}[section]

\newtheorem{lemma}[theorem]{Lemma}
\newtheorem{proposition}[theorem]{Proposition}

\theoremstyle{definition}

\newtheorem{example}[theorem]{Example}

\theoremstyle{remark}
\newtheorem{remark}[theorem]{Remark}

\numberwithin{equation}{section}

\newcommand{\N}{\mathbb{N}}
\newcommand{\R}{\mathbb{R}}

\newcommand{\mX}{\mathbb{X}}

\newcommand{\Z}{\mathbb{Z}}

\newcommand{\p}{\mathbb{P}}

\newcommand{\cB}{\mathcal{B}}
\newcommand{\cC}{\mathcal{C}}
\newcommand{\cD}{\mathcal{D}}

\newcommand{\cG}{\mathcal{G}}
\newcommand{\cF}{\mathcal{F}}

\newcommand{\cK}{\mathcal{K}}

\newcommand{\cN}{\mathcal{N}}

\newcommand{\cO}{\mathcal{O}}
\newcommand{\cP}{\mathcal{P}}

\newcommand{\cU}{\mathcal{U}}
\newcommand{\cQ}{\mathcal{Q}}

\newcommand{\fD}{\mathfrak{D}}

\newcommand{\fS}{\mathfrak{S}}

\newcommand{\E}[1]{\mathbb{E}\left [  #1 \right ]}
\newcommand{\EE}[1]{\mathbb{E} [  #1 ]}

\newcommand{\V}[1]{\operatorname{Var}\left [  #1  \right ] }

\renewcommand{\epsilon}{\varepsilon}
\renewcommand{\phi}{\varphi}

\newcommand{\intd}[1]{\mathrm{d}#1}

\newcommand{\1}[1]{\,\mathds{1}\! \left\{ #1 \right\} }
\newcommand{\cov}[2]{\operatorname{Cov} \left(#1,#2		\right) }

\newcommand{\poi}{\operatorname{Poi}}
\newcommand{\diam}{\operatorname{diam}}

\newcommand{\ol}{\overline}
\newcommand{\wt}{\widetilde}
\usepackage{color}
\usepackage{eso-pic}
\definecolor{lightgray}{gray}{.80}

\allowdisplaybreaks

\begin{document}

\title{
On the law of the iterated logarithm and strong invariance principles in stochastic geometry
\footnote{This research was partially supported by the German Research Foundation (DFG), Grant Number KR-4977/1-1.}
}

\author{Johannes Krebs\footnote{Department of Mathematics, TU Braunschweig, 38106 Braunschweig, Germany,\ email: \tt{johannes.krebs@tu-braunschweig.de} }\; \footnote{Corresponding author}
}

\date{\today}
\maketitle

\begin{abstract}
We study the law of the iterated logarithm (\cite{khinchin1924asymptotische}, \cite{andrei1929kolmogorov}) and related strong invariance principles for functionals in stochastic geometry. As potential applications, we think of well-known functionals defined on the $k$-nearest neighbors graph and important functionals in topological data analysis such as the Euler characteristic and persistent Betti numbers.
\setlength{\baselineskip}{1.8em}

\end{abstract}

In the present manuscript, we study the law of the iterated logarithm (LIL) and strong invariance principles (SIP) in stochastic geometry (\cite{penrose2003random}) with a particular focus on its applications in topological data analysis (TDA), for the latter we refer to \cite{zomorodian2005computing}, \cite{carlsson2009topology}, \cite{edelsbrunner2010computational}.

The original statement of the LIL is due to \cite{khinchin1924asymptotische}; another version is due to \cite{andrei1929kolmogorov}. The nowadays more general statement was proved by \cite{hartman1941law}. Let $(X_i)_{i\in\N}$ be independent and identically distributed with mean zero and unit variance, then
\[
			\limsup_{n\to\infty} \frac{1}{\sqrt{2n\log\log n}}\sum_{i=1}^n X_i = 1 \quad a.s.
\]
\cite{strassen1964invariance} introduced the striking functional form of the LIL which has then been recovered in various functional settings, in particular, martingales, see \cite{heyde1973invariance}, \cite{wichura1973some},  \cite{philipp1977functional}, \cite{hall1980martingale}. 
\cite{stout1970hartman} gives another variant of the LIL for martingales.

In the context of dependence the LIL has been investigated for mixing processes with a fast decaying correlation structure as well as for stationary processes, see \cite{philipp1969law}, \cite{oodaira1971law}, \cite{rio1995functional} and \cite{schmuland_sun_2004} as well as \cite{wu2007strong} and \cite{zhao2008law}. \cite{li1992law} and \cite{jiang1999some} study the LIL for random fields. The applications of the LIL in statistical tests and U-statistics were studied by \cite{arcones1995law}, \cite{dehling2009law}. For empirical processes the LIL was studied by \cite{arcones1997law}.  

Strong invariance principles (\cite{komlos1975approximation, komlos1976approximation}) consider the asymptotic approximability of the functional of interest by a Brownian motion. In this context, many deep results have been obtained for time series under various dependence structures, see \cite{philipp1975almost}, \cite{eberlein1986strong}, \cite{wu2007strong}, \cite{berkes2014KMT}.

The present contribution of the LIL and the SIP for functionals in stochastic geometry and TDA relies on the idea of strong stabilization, which dates back to \cite{lee1997central, lee1999central} and which has been popularized in the central limit theorem of \cite{penrose2001central}. We refer to \cite{chazal2017introduction} for an introduction to TDA and persistent homology.

We consider so-called stabilizing and translation invariant functionals $H$ defined on point processes such as the homogeneous Poisson process on $\R^d$. Let $(\mathbb{X}_n)_n$ be a sequence of finite subsamples of such a point process, specified in detail below. We show
\begin{align*}
				\limsup_{n\to\infty} \frac{1}{\sqrt{2n\log\log n}} ( H(\mathbb{X}_n ) - \mathbb{E}[H(\mathbb{X}_n )] ) = \sigma \quad a.s.,
\end{align*}
where $\sigma^2 = \lim_{n\to\infty} n^{-1} Var( H(\mathbb{X}_n ) ) > 0$. Technically our proofs rely on the observation of W. Philipp, who conjectured, the LIL holds for ``any process for which the Borel-Cantelli Lemma, the central limit theorem with a reasonably good remainder and a certain maximal inequality are valid" (\cite{philipp1969law}). Using an approach, which relies on martingale differences, we show that the functional $H(\mathbb{X}_n ) - \mathbb{E}[H(\mathbb{X}_n )]$ can indeed be decomposed in a main term consisting of a partial sum process w.r.t. to a stationary random field and a remainder which behaves ``reasonably" nice given the stabilizing properties of the functional $H$. We consider two point processes as input, namely, the Poisson process and the binomial process in the critical (thermodynamic) regime, see also \cite{yogeshwaran2017random}. Essentially, this setting means that we observe the point process on an increasing sequence of windows $(W_n)_n\subset\R^d$ with $\cup_n W_n = \R^d$.

Moreover, we consider the special case of Poisson and binomial processes defined on a domain $\cD\times [0,n]$, which stretches in only one direction. Apart from recovering the LIL also in this case, we establish a strong invariance principle. More precisely, there is a standard Brownian motion on an enlarged probability space such that 
$$
		H(\mathbb{X}_n)-\mathbb{E}[H(\mathbb{X}_n)] = B(n \sigma^2 ) + r_n,
$$
where up to logarithmic factors and under suitable moment assumptions the remainder $r_n$ is of order $\cO_{a.s.}( n^{1/4} )$ in the Poisson sampling scheme and of order $\cO_{a.s.}( n^{1/4 + 1/p})$ in the binomial sampling scheme, for a certain $p\in\R_+$.  The technique relies on the classical Skorokhod embedding for martingales, see \cite{strassen1964invariance} and \cite{strassen1967almost}. So in the Poisson case and given the technique of the Skorokhod embedding, the rate is maximal (up to a logarithmic factor). In the binomial case, the rate suffers somewhat from the applied Poisson approximation. Compared to the celebrated KMT approximation, the rates are suboptimal. Improving the rates in the SIP and the generalization to a domain $W_n$, which stretches in each dimension, can be considered as an independent problem, which is probably very difficult. We refer to the discussion at the end of Section~\ref{Sec_MainResults}.

In practice, the quantitative stabilizing properties of the functional $H$ need to be checked of course individually. These are however quite immediate for many functionals if we think of $k$-nearest neighbor problems or the Euler characteristic. The latter being the oldest and simplest descriptor of the topological properties of (point cloud) data. For persistent Betti numbers an exact quantification of the stabilizing properties is still open, even though there are first results (\cite{krebs2019asymptotic}). For the special case of 0-dimensional features, i.e., connected components, see \cite{chatterjee2017minimal}, who quantify the stabilization for the minimal spanning tree.

The methods of TDA have been successfully applied in the past to various fields such as finance (\cite{gidea2018topological}), material science (\cite{lee2017quantifying}) or biology (\cite{yao2009topological}). For recent applications involving the Euler characteristic, see \cite{adler2008some}, \cite{decreusefond2014simplicial}, \cite{crawford2016functional}. Multivariate or functional central limit theorems for the Euler characteristic where proved in \cite{hug2016second}, \cite{thomas2019functional}. Ergodic theorems for the Euler characteristic were established in \cite{schneider2008stochastic}. 
The present contribution is not limited to establishing limit theorems in TDA but also lays a certain groundwork of statistically sound testing procedures in the context of stochastic geometry, see also \cite{robbins1970statistical}, \cite{lerche1986sequential}, \cite{wu2005bahadur} for potential applications.

The rest of the paper is organized as follows. We introduce the basic notation and the stochastic model in Section~\ref{Sec_Preliminaries}. The main results are stated in Section~\ref{Sec_MainResults}. The technical details are given in Section~\ref{Section_Proofs}, some of them are deferred to the Appendix~\ref{Appendix}.

\section{Preliminaries}\label{Sec_Preliminaries}

Let $H$ be a real-valued functional defined for all finite subsets of $\R^d$. We assume that $H$ is translation invariant, i.e., $H( P + x) = H(P)$ for all finite $P\subset \R^d$ and $x\in\R^d$. Classical results on limit theorems in this context rely on the add-one cost function of $H$ which is defined for a finite set of points $P\subset\R^d$ by
\[	\fD_0( P) \coloneqq H( P \cup \{0\} ) - H(P).
\]
Let $\cP$ be a homogeneous Poisson process of unit intensity on $\R^d$. The functional $H$ is strongly stabilizing on $\cP$ if there exist $a.s.$ finite random variables $S$ (the so-called radius of stabilization) and $\fD_0(\cP,\infty)$ such that
\begin{align}\label{D:StrongStab}
		\fD_0( (\cP\cap B(0,S)) \cup A ) = \fD_0( \cP,\infty) \text{ with probability 1}
\end{align}
for all finite $A\subset \R^d\setminus B(0,S)$, where $B(z,r)$ is the closed $r$-neighborhood of $z$ w.r.t. the Euclidean distance $d$ on $\R^d$ for $z\in\R^d$ and $r\in\R_+$. 

Before we introduce the models, we begin with some general notation. Let $(S,\fS,\mu)$ be some generic measure space and $q\ge 1$. Let $f\in L^q(S,\fS,\mu)$. We denote its $q$-norm by $\|f\|_q$. Given a $d$-dimensional point cloud $P$ and $z\in\R^d$, we write $P+z$ for the translated point cloud $\{p+z: p\in P\}$. For a set $A\subset\R^d$ and $\delta\in\R_+$, we write $A^{(\delta)}$ for the collection of all points with a distance to $A$ of at most or equal to $\delta$. For a finite set $A$ we write $\# A$ for its cardinality. If $A\subset\R^d$ is infinite, we write $|A|$ for its Lebesgue measure. The diameter of $A$ w.r.t.\ the Euclidean distance is abbreviated by $\diam(A)$. 
Let $A,B\subset\R^d$, then we write $d(A,B) = \inf\{ d(x,y): x\in A, y\in B\}$.

We will use a sequence of observation windows given by $W_n = [-n^{1/d}/2,n^{1/d}/2]^d$, $n\in\N$. (A generalization to other observation windows is a routine, however, in the present contribution, we do not treat this any further.) The topological boundary of $W_n$ is $\partial W_n$. Let $Q_0 = [-1/2, 1/2]^d$ and $Q_z = Q_0 + z$, $z\in\R^d$. Moreover, we define the index set $B_n = \{z\in\Z^d: Q_z \cap W_n \neq \emptyset \}$. Then $\# B_n / n \to 1$ as $n\to \infty$.
Also set $D_n = \{ z\in\Z^d: \| z\|\le n \}$, where $\|\cdot\|$ is the $\infty$-norm on $\Z^d$.

Let $\cP'$ be another homogeneous Poisson process of unit intensity on $\R^d$ independent of $\cP$. Write $\cP''_z$ for $(\cP \setminus Q_z) \cup ( \cP' \cap Q_z)$, $z\in\R^d$. We write $\cP_n$ ($\cP'_n$, $\cP''_{n,z}$) for the restriction of $\cP$ ($\cP'$, $\cP''_z$) to $W_n$. 

Furthermore, for each $n,m\in\N$, we introduce a binomial process $\cU_{n,m}$ of length $m$ on $W_n$, whose elements have a uniform distribution on $W_n$. We use the following coupling. Set $N_n = \cP( W_n )$. $N_n$ is Poisson distributed with parameter $n$ and $N_n \le N_{n+1}$ for each $n\in\N$. Let the Poisson process $\cP_n$ be given by $X_1,\ldots,X_{N_n}$.  Let $V_{n,1},V_{n,2},\ldots$ be i.i.d.\ on $W_n$. Define the $n$th binomial process on $W_n$ by
\begin{align}
\begin{split}\label{E:CouplingBinom}
		\cU_{n,m} &= \{ X_{i}: i \in \{1,\ldots,m \wedge N_n \} \} \\
		&\qquad \cup \{ V_{n,k }: k \in \{1,\ldots, (m-N_n)\wedge 0 \} \}, \quad m\in\N.
	\end{split}
\end{align}
Write $\preceq$ for the lexicographic ordering on $\Z^d$. For $z\in\Z^d$, let $\cF_z$ be the smallest $\sigma$-field such that the Poisson points of $\cP$ in $\cup_{y \preceq z} Q(y)$ are measurable, that is $\cF_z = \sigma( \cP|_{Q(y)}: y\in\Z^d, y \preceq z )$.

We study two models which have different stabilizing and growth properties. The functional $H$ satisfies a hard-thresholded stabilizing condition in model~\ref{itm:model1}, e.g., the Euler characteristic. In model~\ref{itm:model2}, the functional $H$ is exponentially stabilizing but satisfies a polynomial growth condition, for functionals with these properties see \cite{penrose2001central}.
\begin{itemize}
    \item [\namedlabel{itm:model1}{(M1)}] \textit{Hard-thresholded stabilization.} Let $P$ be a finite point cloud. Then $H$ satisfies \eqref{D:StrongStab}, where $S\le S^*$ $a.s.$ for some $S^*\in\R_+$. Moreover, for all locally finite point clouds $P$
    \begin{align}\label{C:GrowthM1}
    	H( P \cap B(z,S^*) ) \le C^* e^{q N(z)},
	\end{align}
	  for some $C^*,q\in\R_+$ and $N(z) = P( Q_z^{(\delta)} )$, where $\delta>0$ is independent of $z\in\R^d$.
    
    \item  [\namedlabel{itm:model2}{(M2)}] \textit{Exponentially stabilizing and polynomially bounded functionals.} 
		The uniform bounded moments condition is given in terms of the binomial process. Given a set $A$ from the class of sets 
		$$\cB=\{ W_n + z, B(z,w), (W_n + z) \cap B(z,w) : n\in \N, w\in\R_+, z\in\R^d \},$$
let $\cU_{A,m}$ be a binomial process of length $m$ on $A$. We require a uniform bounded moments condition to be satisfied. The order of the moment is characterized by two parameters and differs depending on the sampling scheme as follows. For \textit{Poisson sampling} let $p> 2d \vee 4$ and $r :\equiv \infty$ (so that $r/(r-1) = 1$). For \textit{Binomial sampling} let $p\in\N$, $p$ even, such that $p> 2d \vee 6$ and $r\in (1, (2p-4)/(p+2) )$. Then
		\begin{align}\label{C:UniformBoundedMoments}
			\sup_{A\in\cB, 0\in A }	\sup_{m\in [1/2 |A|, 3/2 |A| ]} \E{  | 				H(\cU_{A,m} \cup \{0\} ) - H(\cU_{A,m}) |^{p r/(r-1)  } } < \infty
		\end{align}
Furthermore, the functional $H$ is polynomially bounded in the sense that $|H(P)| \le \nu ( \diam(P) + \# P )^\nu$ for a constant $\nu\in\R_+$.

Denote $\cC_{y,n}$ the cube with center $y$ and edge length $n^{1/d}$, $y\in\R^d$ and $n\in\N$. Let $V_1,V_2,\ldots,$ be a sequence of i.i.d. random variables with uniform distribution on $Q_0$. Let $x\in Q_0$, $k,n\in\N$ and $y\in\R^d$. Consider the point cloud $P(y,n,k,x) =  ( (\cP|_{\cC_{y,n} }\setminus Q_0 ) \cup \{V_1,\ldots,V_k\} ) - x$. There is an $a.s.$ finite random radius $\wt S = \wt S(y,n,k,x)$ such that
\begin{align}\begin{split}\label{E:GeneralStabilization}
			&\fD_0(  P(y,n,k,x) \cap B(0,\wt S) ) = \fD_0( [  P(y,n,k,x) \cap B(0,\wt S) ]   \cup A )
\end{split}\end{align}
for all finite $A \subset \R^d\setminus B(0,\wt S)$ and $\wt S$ satisfies: There are constants $c_1,c_2>0$ with
\begin{align}\label{C:ExpStabilization}
			\sup \ \p( \wt S_{y,n,k,x} > r ) \le c_1 \exp( -c_2 r  ), 
			\quad r \in (0, \infty),
\end{align}
where the supremum is taken over all possible parameters $x\in Q_0, k,n\in\N, y\in\R^d$.
\end{itemize}

The stabilizing condition in model~\ref{itm:model1} features some kind of $m$-dependence, we will see this later in the proofs. Similarly, the exponential stabilizing condition in \ref{itm:model2} features a certain analogy to exponential mixing conditions.

One finds that the growth condition from \eqref{C:GrowthM1} and the hard-thresholded stabilization in model~\ref{itm:model1} imply the uniform bounded moments condition \eqref{C:UniformBoundedMoments} in model~\ref{itm:model2} for each $p\in\N$. However, in contrast to model~\ref{itm:model2}, we do not assume an overall polynomial bound in model~\ref{itm:model1}.

The order of the uniform bounded moments condition varies depending on the model. In the Poisson sampling scheme, the order is greater than $2d\vee 4$, whereas it is significantly greater in the binomial sampling scheme. It is very likely that this latter bound is suboptimal. However, improving the bounds without additional assumptions is probably very difficult and can be considered as an independent problem. Technically, the results for the binomial process suffer from the applied Poissonization arguments.

The CLT of \cite{penrose2001central} for strongly stabilizing and polynomially bounded functionals $H$ requires a similar uniform bounded moments condition to be satisfied with $p=4$. In order to establish the LIL and and the SIP, we need to quantify the impact of the moment condition on the order of the approximation.

In the context of normal approximation, many functionals have been extensively studied under various stabilizing conditions, see \cite{lachieze2019normal}. Popular examples are functionals defined the $k$-nearest neighbor graph (\cite{bickel1983sums}) or functionals defined on Voronoi tessellations (\cite{schulte2012voronoi}, \cite{thaele2016voronoi}). In this contribution, we study the Euler characteristic of simplicial complexes and the total edge length of the $k$-nearest neighbor graph. 

\begin{example}[Euler characteristic]\label{Ex:EC}
Let $P\subset\R^d$ be a finite point cloud. Let $\cK$ be a simplicial complex obtained from $P$ by a rule which puts a uniform upper bound on the diameter of a simplex, e.g., the {\v C}ech or the Vietoris-Rips complex for a certain filtration parameter $r\in\R_+$. Then the Euler characteristic of $\cK$ is	$\chi = \sum_{k=0}^\infty (-1)^k S_k( \cK )$, where $S_k$ is the number of $k$-dimensional simplices in $\cK$. Similar to (persistent) Betti numbers, the Euler characteristic is a topological invariant, see \cite{hatcher}. Since the diameter of each simplex in $\cK$ is uniformly bounded above by some $r\in\R_+$, when adding a point $z$ to the point cloud $P$, only those points in the $r$-neighborhood of $z$ can form new simplices involving $z$. Hence, the Euler characteristic enjoys the hard-thresholded stabilization of model~\ref{itm:model1}. Moreover, using the rule that $m$ points can form at most $\binom{m}{k+1}$ many $k$-simplices for $0\le k\le m-1$, one sees that also the growth conditions of model~\ref{itm:model1} are satisfied, see also Section~\ref{Section_VerificationExample}.
\end{example}

\begin{example}[Total edge length in the $k$-nearest neighbor graph]\label{Ex:kNN}
Let $H$ be the total edge length obtained from the $k$-nearest neighbor graph of a set of points. \cite{penrose2001central} show the asymptotic normality of this functional, they also give details on the uniform bounded moments condition. Rates of convergence in the Kolmogorov distance are given in \cite{lachieze2019normal}, here one also needs to employ an exponential stabilizing condition similar as in \eqref{C:ExpStabilization}. We show in the Section~\ref{Section_VerificationExample} that the total edge length $H$ satisfies the uniform stabilization property from model~\ref{itm:model2} on the homogeneous Poisson process $\cP$.
\end{example}

\section{Main results}\label{Sec_MainResults}

In this section we study the rate of convergence of $n^{-1}(H(\mX_n)-\E{ H(\mX_n)})$. In general for stationary and ergodic processes $(X_n)_n$, their empirical mean $n^{-1}\sum_{i=1}^n X_i$ can converge to $\E{X_0}$ arbitrarily slowly (see \cite{krengel2011ergodic}). However, given appropriate moment conditions and decay rates for the dependence structure of the stochastic model as in \ref{itm:model1} and \ref{itm:model2}, one obtains the optimal rates known from i.i.d.\ data, which are given in terms of the law of the iterated logarithm.

In the following, we will need the notion of the random variables $\Delta(z,\infty)$, $z\in\Z^d$, which satisfy together with the radius of stabilization $S_z$, say, the requirement
\begin{align}\begin{split}\label{D:S0}
				 &H( [\cP \cap B(z,S_z)]\cup A ) - H( [\cP''_z \cap B(z,S_z)]\cup A ) \\
				 &\qquad\qquad = H( \cP \cap B(z,S_z) ) - H( \cP''_z \cap B(z,S_z) ) =: \Delta(z,\infty) \quad a.s.
\end{split}\end{align}
for all $A\subset \R^d\setminus B(z,S_z)$. The existence of $\Delta(z,\infty)$ follows from the assumptions in both models, see \cite{penrose2001central} for a rigorous proof. We state the LIL for both models~\ref{itm:model1} and \ref{itm:model2} for an underlying sequence of Poisson processes $(\cP_n)_n$.

\begin{theorem}[LIL for the Poisson process]\label{T:LILPoisson}
Assume that $\fD_0(\cP,\infty)$ is nondegenerate. Let $\sigma = \sqrt{ \mathbb{E}[ \mathbb{E}[ \Delta(0,\infty) | \cF_0 ] ^2 ] }$ which is positive. Then, for either choice of sign,
\begin{align*}
		\limsup_{n\to\infty}  \pm \frac{1}{\sqrt{ 2 n \log \log n}} ( H(\cP_n) - \E{H(\cP_n)} ) = \ \sigma \quad a.s.
\end{align*}
\end{theorem}

In the same spirit, we obtain the LIL for the sequence of binomial processes $(\cU_{n,n})_n$.

\begin{theorem}[LIL for the binomial process]\label{T:LILBinomial}
Assume that $\fD_0(\cP,\infty)$ is nondegenerate. Then, for either choice of sign,
\begin{align*}
		\limsup_{n\to\infty} \pm \frac{1}{\sqrt{ 2 n \log \log n}} ( H( \cU_{n,n} ) - \E{H( \cU_{n,n} )} ) = \tau \quad a.s.,
\end{align*}
where $\tau = \sqrt{ \sigma^2 - \alpha^2 }$ is positive for $\sigma = \sqrt{ \mathbb{E}[ \mathbb{E}[ \Delta(0,\infty) | \cF_0 ] ^2 ] }$ and $\alpha = \E{ \fD_0 (\cP,\infty) }$.
\end{theorem}

The key idea of Theorem~\ref{T:LILPoisson} is to approximate $(H(\cP_n)-\EE{H(\cP_n)})_n$ by the stationary process $(F_z)_{z\in B_n}$, where $F_z = \E{ \Delta(z,\infty) | \cF_z }$, and to show that the difference is of order $o(n^{1/2}\log\log n)$ with probability 1. Then we apply standard techniques of the LIL for stationary random fields which are for instance also considered in \cite{schmuland_sun_2004}. In order to obtain Theorem~\ref{T:LILBinomial}, we rely on the classical Poissonization trick and couple $H(\cU_{n,n})$ with $H(\cP_n)$ at an accuracy of $o(n^{1/2}\log\log n)$.

We compare our results to the case of the LIL for a causal stationary process $(X_i)_i$ with $X_i=g(\xi_i)$, where $\xi_i = (\epsilon_i,\epsilon_{i-1},\ldots)$ and where $g$ is a real-valued Borel-measurable function aggregating index dependent a part of the i.i.d.\ sequence $(\epsilon_j)_j$. In this setting the variance expression is very similar, see, e.g., \cite{wu2007strong}. The deeper connection between the causal stationary process and the current setting becomes most apparent if we consider the leading term of the Poisson sampling scheme $H(\cP_n)-\E{H(\cP_n)}$, which is 
$$
	\sum_{z\in B_n} \E{ \Delta(z,\infty) | \cF_z } = \sum_{z\in B_n} \E{ H(\cP\cap B(z,S_z)) - H(\cP''_{z} \cap B(z,S_z)) | \cF_z }.
$$
For the causal stationary process, the leading term of the partial sum $\sum_{k=1}^n X_k$ is
$
	\sum_{k=1}^n \mathbb{E}\Big[ \sum_{i\ge k} g(\xi_i) - g(\xi'_i) | \cF_k \Big],
$
where $\xi'_i = (\epsilon_i,\ldots,\epsilon_1,\epsilon'_0, \epsilon_{-1},\ldots)$ for an i.i.d.\ copy $\epsilon'_0$, see \cite{wu2007strong}. Essentially, the increasing index sets combined with the stationarity lead now to a very similar behavior of the two partial sums which in turn ultimately leads to the same results.

\subsection{The special case of a one-dimensional domain}\label{OneDimDomain}
In this section, we consider a slight modification of the above model. Instead of using the observation windows $(W_n)_n$, we observe the Poisson or the binomial process on a cylinder-like domain $\wt W_n = \cD \times [0,n]$, which stretches only in one direction. Here $\cD\subset \R^{d-1}$ is a bounded and convex set. In this case, we can easily recover the LIL for both models. Moreover,  using a Skorokhod embedding, we obtain a strong invariance principle. 

We need to adjust slightly the definitions to the new domain, that is all quantities describing the models \ref{itm:model1} and \ref{itm:model2} are defined in terms of the domain $\cD\times \R$, in particular, the add-one cost function $ \fD_0 = \fD_0(\cP|_{\cD\times \R} ,\infty)$, the radius of stabilization $S$ and the random variable $\Delta(0,\infty) = \Delta_{\cP|_{\cD\times \R} }(0,\infty)$. Also all point processes $P = P|_{\cD\times \R}$ are restricted to this domain. The $\sigma$-fields are adjusted, we use $\wt \cF_z = \sigma( \cP|_{\cD \times (-\infty,1/2 + y] }: y\in\Z, y\le z)$ instead. 

A particular interesting functional $H$ in this framework is the persistent Betti number obtained from the Vietoris-Rips or the {\v C}ech complex of a point cloud in $\cD\times\R$. It is well-known that the persistent Betti number is polynomially bounded and satisfies the uniform bounded moments condition for each $p\in\N$, this follows comparably directly from calculations similar as in \cite{yogeshwaran2017random}. Moreover, since the domain $\cD$ is bounded in this 1-dimensional setting, we do not observe percolation effects and the radius of stabilization decays at an exponential rate as required for model~\ref{itm:model2}, see also \cite{krebshirsch2020} for more details.

\begin{theorem}[1-dimensional LIL]\label{T:1DLIL}
Let either $\cQ$ be a Poisson process of unit intensity on $\cD\times \R$ and $\cQ_n = \cQ|_{\wt W_n}$ or let $\cQ_n$ be a binomial process of length $n$ and uniform distribution on $\wt W_n$. If $\fD_0(\cQ|_{\cD\times \R} ,\infty)$ is nondegenerate, then, for either choice of sign,
\begin{align*}
		\limsup_{n\to\infty} \pm  \frac{1}{\sqrt{ 2 n \log \log n}}( H(\cQ_n) - \E{H(\cQ_n)} ) = \ \nu \quad a.s.,
\end{align*}
where $\nu^2$ is $\mathbb{E}[ \mathbb{E}[ \Delta_{\cQ|_{\cD\times \R} }(0,\infty)| \wt \cF_0]^2 ]$  in the Poisson model and $\mathbb{E}[ \mathbb{E}[ \Delta_{\cQ|_{\cD\times \R} }(0,\infty)|\wt \cF_0]^2 ] - \mathbb{E}[\fD_0(\cQ|_{\cD\times \R}, \infty)]^2 $ in the binomial model.
\end{theorem}
In contrast to Theorem~\ref{T:LILPoisson} and \ref{T:LILBinomial}, we consider here a domain with one end being fixed, so in the limit we study the process on the half space $\cD\times [0,\infty)$ and not on $\cD\times\R$. The statement can however be carried over to this setting very directly because the applied reasoning is invariant under translations, see also the corresponding proofs.

The strong invariance principle studied in this case means that the (discrete) path $n\mapsto H(\cQ_n)-\EE{H(\cQ_n)}$ can be approximated by Brownian motion. Consequently, asymptotic properties for the process of interest can be obtained from those of Brownian motion (up to a certain approximation error).

In order to state strong invariance principles, one usually needs to work on an enlarged probability space, which is rich enough to contain all those random variables necessary for the approximation. Also one needs to redefine the process on this probability space without changing its distribution. So, for brevity, we say, there is a richer probability space and a standard Brownian motion $B$, such that $n\mapsto H(\cQ_n)-\EE{H(\cQ_n)}$ can be approximated by $B$, see also \cite{wu2007strong} for this convention.

\begin{theorem}[1-dimensional SIP]\label{T:1DSIP}
Let the assumptions of Theorem~\ref{T:1DLIL} be satisfied. In model~\ref{itm:model2}, let $p\ge 8$ additionally to the conditions which lead to \eqref{C:UniformBoundedMoments}. Then on a richer probability space, there is a standard Brownian motion $B$ such that
\[
		H(\cQ_n)-\mathbb{E}[H(\cQ_n)] = B(n \nu^2 ) + \cO_{a.s.}( n^{1/4} (\log n)^{1/2} (\log\log n)^{1/4} )
		\]
		if $(\cQ_n)_n$ is the sequence of Poisson processes and
	\[
		H(\cQ_n)-\mathbb{E}[H(\cQ_n)] = B(n \nu^2 ) +\cO_{a.s.}( n^{1/4 + 1/p} (\log n)^2 )
		\]
			if $(\cQ_n)_n$ is the sequence of binomial processes.
\end{theorem}

Strong invariance principles are of considerable importance in probability theory and have received attention in the statistical inference of dependent processes. Motivated by Strassen's work (\cite{strassen1964invariance, strassen1967almost}) approximating partial sums by Brownian motions have been considered for i.i.d.\ sequences, martingale differences and stationary ergodic processes. The celebrated KMT approximation (\cite{komlos1975approximation, komlos1976approximation}) shows that for an i.i.d.\ sequence of random variables with finite $p$th moment, $p>2$, the optimal approximation rate in the SIP is $o(n^{1/p})$. Strong invariance principles for partial sums under dependence have been studied by \cite{philipp1975almost}, \cite{berkes1979approximation}, \cite{eberlein1986strong}, \cite{shao1993almost}, \cite{wu2007strong}
Regarding statistical applications, the SIP has been considered in change-point and trend analysis (\cite{csorgo1997limit}). \cite{berkes2014KMT} show that the optimal rates in the KMT approximation are also valid for time series under certain dependence conditions.

In the present models~\ref{itm:model1} and \ref{itm:model2}, it is unclear whether a bound $o(n^{1/p})$ can be obtained for general $p$ and for the above sequence of observation windows $(W_n)_n$ which stretches in each dimension. The application of Strassen's method based on the Skorokhod embedding appears to be limited to the case of the stretched domain $\cD\times [0,n]$ and limits the approximation property to $n^{1/4}$ up to logarithmic factors. In the binomial setting the rate suffers additionally from the fact that a Poissonization argument is applied.

Whether the tools laid out in \cite{berkes2014KMT} can lead to a SIP with the optimal rate for the general sequence $(W_n)_n$ is an open question. Here the findings of  \cite{bierme2014IP}, who investigate invariance principles for random fields, are also of interest. We leave this probably very difficult problem up to future research.

\section{Technical results}\label{Section_Proofs}
We use the following notation and abbreviations throughout the rest of this manuscript.\vspace{.5em}

\textit{Convention about constants.} To ease notation, most constants in this paper will be denoted by $C,c,c'$, etc. and their values may change from line to line. These constants may depend on parameters like the dimension and often we will not point out this dependence explicitly; however, none of these constants will depend on the index $n$, used to index infinite sequences, or on the index, which is used to index martingale differences.\vspace{.5em}

\textit{The Poisson process.} $\cP$ and $\cP'$ are independent homogeneous Poisson processes of unit intensity on $\R^d$. Let $z\in\Z^d$. $S_z$ is the random variable which satisfies
\begin{align*}
	& \Delta(z,n) \coloneqq H(\cP_n ) - H( \cP''_{n,z} ) \\
	&\qquad\qquad \to \Delta(z,\infty) \coloneqq H(\cP \cap B(z,S_z) ) - H( \cP''_{z} \cap B(z,S_z) ) \; a.s. \quad  (n\to\infty).
\end{align*}
Moreover, we define for $z\in\Z^d$
\begin{align*}
	\Delta'(z,n) & \coloneqq   \cP(W_n) -  \cP''_{z}(W_n)	\to	\Delta'(z,\infty) \coloneqq  \cP (Q_z) -  \cP' (Q_z) \; a.s. \quad (n\to\infty).
\end{align*}
Finally, put $F_z = \E{ \Delta(z,\infty)|\cF_z }$ and $F'_z = \E{ \Delta(z,\infty) - \Delta'(z,\infty)|\cF_z }$, $z\in\Z^d$. Note that $(F_z)_z$ and $(F'_z)_z$ are both stationary processes.\vspace{.5em}

\textit{The coupled binomial processes.} The binomial and the Poisson process are coupled as in \eqref{E:CouplingBinom}. We write $I_n$ for the set which contains the integers between $n$ and $N_n$, i.e.,
\begin{align}\label{E:LILBinomialError1b}
	I_n = \{n,\ldots,N_n-1\} \cup \{N_n,\ldots,n-1\},
\end{align}
where the first set or the second set on the right-hand side is empty.\vspace{.5em}

\subsection{A general LIL}
We begin with three general statements regarding the stationary process $( F_z: z\in\Z^d )$. The proofs are very similar to those given in \cite{schmuland_sun_2004}, who study the LIL for random fields under modified $\phi$-mixing conditions (which do not apply to our setting), and are deferred to the Appendix~\ref{A:LIL}.

For Lemma~\ref{L:Covariance}, Propositions~\ref{P:NormalApproximation}, \ref{P:MaximalInequality} and Theorem~\ref{T:GeneralLIL}, let $h\colon \R\to\R$ be either $h=\text{id}$ or $h=(\cdot)^2 - \mathbb{E}[ F_0^2 ]$. Then in both cases $\mathbb{E}[ h(F_0) ] =0$. Moreover we require in this section, $\mathbb{E}[h( F_0 )^{4} ] < \infty$. When deriving the LIL, $h$ is the identity, and the latter condition is satisfied in both models~\ref{itm:model1} and \ref{itm:model2}. When deriving the SIP, $h$ also equals $(\cdot)^2 - \mathbb{E}[ F_0^2 ]$, in this case, we will then need $\EE{ F_0^{8} }< \infty$, see also the requirements in Theorem~\ref{T:1DSIP}.  

First, we define the quantity $\wt \sigma^2 \coloneqq \sum_{z\in\Z^d}\mathbb{E}[h( F_0 ) h(F_z) ]$. We have
\begin{lemma}[Covariance]\label{L:Covariance}
The definition of $\wt \sigma^2$ is meaningful for $h=\text{id}$ and $h=(\cdot)^2 - \mathbb{E}[ F_0^2 ]$.
Moreover, let $X = G( \sum_{z\in I} h(F_z) )$ and $Y = \wt G( \sum_{z\in J} h(F_z) )$ for two bounded Borel functions $G, \wt G:\R\to\R$. Write $d(I,J) = \min\{ \|z-z'\|: z\in I, z'\in J\}$. Then there are constants $c_1,c_2\in \R_+$, which do neither depend on $I$ nor on $J$ such that
\[
		\cov{X}{Y} \le c_1 ( \# I +\# J ) \exp( - c_2 d(I,J) ).
\]
\end{lemma}

\begin{proposition}[Normal approximation]\label{P:NormalApproximation}
 Let $\Phi$ be the distribution function of the standard normal distribution. Let $\epsilon\in (0,1/2)$, then
\[
		\sup_{z\in\R_+} \Big| \p\Big( \sum_{z\in D_n } h(F_z) \le \wt \sigma (\#  D_n) ^{1/2} z \Big) - \Phi(z) \Big | \le C n^{-\epsilon/4}.
\]
\end{proposition}

\begin{proposition}[Maximal inequality]\label{P:MaximalInequality}
Let $\beta>1$. There is a $\rho>0$ such that
\[
			\p\Big( \max_{1\le j\le n} \Big|\sum_{z: \|z\|\le j} h(F_z) \Big| \ge \beta \sqrt{ 2 \wt \sigma^2 (2n+1)^d \log \log n } \Big) \le C (\log n )^{-(1+\rho)}.
\]
\end{proposition}

\begin{theorem}[The LIL]\label{T:GeneralLIL} Given the present assumptions, for either choice of sign,
$$
		\limsup_{n\to\infty} \pm (2n\log\log n)^{-1/2} \sum_{z\in B_n} h( F_z ) =  \wt \sigma,$$
viz.,
\begin{align*}
	&	\limsup_{n\to\infty} \pm \frac{1}{\sqrt{2n\log\log n}} \sum_{z \in B_n} \E{\Delta(z,\infty)|\cF_z } =  \EE{\E{\Delta(z,\infty)|\cF_z }^2}^{1/2}, \\
	&	\limsup_{n\to\infty} \pm \frac{1}{\sqrt{2n\log\log n}} \sum_{z\in B_n}  \Big(\E{\Delta(z,\infty)|\cF_z }^2 - \sigma^2\Big) = \Big(\sum_{z\in\Z^d} \E{ F_z^2 F_0^2 - \sigma^4  } \Big)^{1/2} .
\end{align*}
\end{theorem}

\begin{remark} \label{R:Binomial}
In the same spirit, the three statements above are valid for the process $(F'_z: z\in \Z^d)$, if $\E{ h(F'_z)^4}<\infty$. We simply replace the variances; if $h$ is the identity, we replace $\sigma^2$ with $\tau^2 = \EE{ \E{ \Delta(0,\infty) - \alpha \Delta'(0,\infty) | \cF_0 }^2 }$, this last equality will be verified below in Lemma~\ref{L:TauSigma}. Note that by definition $\EE{ |\Delta'(0,\infty)|^q } < \infty$ for all $q\in\N$.
\end{remark}

\begin{proof}[Proof of Theorem~\ref{T:GeneralLIL}]
Set $\phi(n) = \sqrt{ 2 \wt \sigma^2 (2n+1)^d \log \log n }$, $n\in\N$, $n>e$. Using $S_n = \sum_{z\in D_n} h(F_z)$, 
\begin{align*}
			\limsup_{n\to\infty} \pm \frac{ \sum_{z\in B_n} h(F_z) }{ \sqrt{ 2 \wt \sigma^2 \# B_n  \log \log \# B_n } } = \limsup_{n\to\infty} \pm \frac{ \sum_{z \in D_n} h(F_z) }{ \sqrt{ 2 \wt \sigma^2 \# D_n \log \log n} } = \limsup_{n\to\infty} \pm \frac{ S_n }{ \phi(n) } .
\end{align*}

Let $\epsilon>0$. First, we show 	$\limsup_{n\to\infty}  | S_n| / \phi(n)  < 1 + \epsilon$ with probability 1.
Set $n_k = \floor{ (1+\tau)^k / \wt \sigma^2 } +1 $ for $\tau>0$. Using Proposition~\ref{P:MaximalInequality},  for each $\gamma>0$ there is a $\rho>0$ such that 
\begin{align*}
		\sum_{k\in\N} \p( \max_{1\le n \le n_k} |S_n| > (1+\gamma) \phi(n) ) \le C \sum_{k\in\N} (\log n_k )^{-(1+\rho)} < \infty.
\end{align*}
Moreover, there is a $k_0\in\N$ such that $\phi(n_k) \le (1+2\tau )^{d/2} \phi( n_{k-1} )$ for all $k\ge k_0$. Also, there are $\gamma\in(0,\epsilon)$ and $\tau> 0$ such that $(1+\epsilon) > (1+\gamma)(1+2\tau)^{d/2}$. Consequently,
\begin{align*}
			\p( \limsup_{n\to\infty} | S_n | / \phi(n) > (1+\epsilon) ) &\le  \p( \limsup_{k\to\infty} \max_{n_{k-1} \le n \le n_k} | S_n | / \phi(n_{k-1} ) > (1+\epsilon) ) \\
			&\le \p( \limsup_{k\to\infty} \max_{ n \le n_k} | S_n | / \phi(n_{k} ) > (1+\epsilon) / (1+2\tau)^{d/2} ) = 0.
\end{align*}
Second, we show $\limsup_{n\to\infty}  | S_n| / \phi(n)  > 1 - \epsilon$ with probability 1. We use two sequences defined by $n_k = k^{4k}$ and $m_k = n_k / k^2$. Consider the event $E_k=E_k (\lambda) = \{ S_{n_k}-S_{m_k} \ge (1-2\lambda)\phi(n_k)\}$. Then
\begin{align}\label{E:LILGeneral1}
		\sum_{k\in\N} \p( E_k(\lambda) ) = \infty
\end{align}
 for all $\lambda>0$. Indeed, consider the inequality
\begin{align}\label{E:LILGeneral2}
		\p( S_{n_k} \ge (1-\lambda)\phi(n_k)) \le \p(E_k) + \p( S_{m_k}\ge \lambda\phi(n_k)).
\end{align}
One finds with the Markov inequality that the second term is negligible in the sense that
\begin{align}\label{E:LILGeneral3}
		\sum_{k=1}^\infty \p( S_{m_k}\ge \lambda\phi(n_k)) < \infty.
\end{align}
 Hence, using \eqref{E:LILGeneral2} and \eqref{E:LILGeneral3}, it is enough to show $\sum_{k=1}^\infty \p( S_{n_k} \ge (1-\lambda)\phi(n_k)) = \infty$ in order to verify \eqref{E:LILGeneral1}. An application of the normal approximation in Proposition~\ref{P:NormalApproximation} yields
\[
		\sum_{k=1}^\infty \big| \p( S_{n_k} \ge (1-\lambda)\phi(n_k)) - \Phi( (1-\lambda) \phi(n_k)/ (\wt \sigma (\# D_{n_k})^{1/2} ) ) \big| \le C \sum_{k=1}^\infty (\# D_{n_k})^{-\delta} <\infty,
\]
where $\delta>0$ and where $\Phi$ is the distribution function of the standard normal distribution. Using the lower bound $1-\Phi(x) \le (2\pi)^{-1/2} x^{-1} \exp(-x^2/2)$, it is a routine to verify for the subsequence $(S_{n_k})_k$
\[
		\sum_{k=1}^n \Phi( (1-\lambda) \phi(n_k)/ (\wt \sigma (\# D_{n_k})^{1/2} ) ) = \sum_{k=1}^n \Phi( (1-\lambda) (2 \log \log (\# D_{n_k}))^{1/2} ) = \infty.
\]
This shows \eqref{E:LILGeneral1}. 

Next, we claim that this leads to $\p( E_k(\lambda) \text{ occurs } i.o. ) = \p( \sum_{k=1}^\infty \1{ E_k(\lambda) }  = \infty ) = 1$ for all $\lambda > 0$. Indeed, we show $\p( \sum_{k=1}^\infty \1{ E_k(\lambda) }  < \infty ) = 0$ with the following considerations
\begin{align*}
		\p\Big( \sum_{k=1}^\infty \1{E_k} \le \frac{1}{2} \sum_{k=1}^n \p(E_k) \Big) &\le \p\Big( \sum_{k=1}^n \1{E_k} \le \frac{1}{2} \sum_{k=1}^n \p(E_k) \Big) \\
		&\le \p\Big( \Big| \sum_{k=1}^n (\1{E_k} - \p(E_k)) \Big| \ge \frac{1}{2} \sum_{k=1}^n \p(E_k) \Big) \\
		&\le \frac{4 Var( \sum_{k=1}^{n} \1{E_k}  ) }{ (\sum_{k=1}^n \p(B_k) )^2 } \\
		&\le \frac{4 (\sum_{k=1}^n \p(E_k) )+ c  }{ (\sum_{k=1}^n \p(B_k) )^2 }  \to 0 \quad (n\to\infty);
\end{align*}
the last inequality is valid because
\begin{align*}
	\sum_{i,j: i < j} | Cov(  \1{E_i} \1{E_j} ) | &\le \sum_{i\neq j} c_1 \exp\Big( -c_2 ( (i+j)^{4(i+j)-2 } - j^{4j}  ) \Big) \\
	&\le  \sum_{i\neq j} c_1 \exp \Big( -c_2 (i+j)^{2}   \Big)  < \infty
\end{align*}
with an application of Lemma~\ref{L:Covariance} and where the second to last inequality follows from a short application of the mean-value theorem. Now,
\[
		E_k( \epsilon/4) \subset \{ S_{n_k} \ge (1-\epsilon) \ \phi(n_k)  \} \cup \{ - S_{m_k} \ge \epsilon/2 \ \phi(n_k) \}.
		\]
One finds $\sum_{k\in\N} \p( |S_{m_k}| \ge \epsilon/2 \ \phi(n_k) ) < \infty$. Thus, $- S_{m_k} / \phi(n_k) \ge \epsilon /2 $ only finitely many times.
Moreover,
\begin{align*}
		1 &= \p( E_k(\epsilon/4) \text{ occurs } i.o. ) \\
		& \le \p( S_{n_k} \ge (1-\epsilon) \ \phi(n_k) \text{ occurs } i.o. ) + \p( - S_{m_k} \ge \epsilon/2 \ \phi(n_k) \text{ occurs } i.o. ).
\end{align*}
This means $\limsup_{k\to\infty} S_{n_k} / \phi(n_k) \ge 1-\epsilon$ with probability 1. 

One can show $\liminf_{k\to\infty} S_{n_k} / \phi(n_k) \le -(1-\epsilon)$ with probability 1 in a similar fashion. This completes the proof.
\end{proof}

\subsection{The LIL for the Poisson process}
\begin{proposition}\label{P:OrderRemainderPoisson}
Assume model~\ref{itm:model1} or \ref{itm:model2} and $p>2$. Then for all $\delta>0$ 
\[
	 \sum_{z\in B_n } \E{ \Delta(z,n) - \Delta(z,\infty) | \cF_z } =  o_{a.s.}( (\log n)^{ 3/2+\delta} \ n^{1/p + (d-1)/(2d)} ).
\]
\end{proposition}
\begin{proof}
We derive a maximal inequality, this will then enable us to prove the claim. We define $Y_{n,z} =  \E{ \Delta(z,n) - \Delta(z,\infty) | \cF_z }$ for $n\in\N$ and $z\in B_n$. Since the $Y_{n,z}$ are martingale differences (w.r.t. the lexicographic ordering), we can apply the Burkholder inequality and the Minkowski inequality
\begin{align}\label{E:LILPoisson2v}
	 \big\|  \sum_{z\in B_n} Y_{n,z} \big \|_p &\le C \big \| \big( \sum_{z\in B_n} Y_{n,z}^2 \big)^{1/2} \big \|_p \le C \big( \sum_{z\in B_n} \big \| Y_{n,z}^2 \big \|_{p/2} \big)^{1/2}.
\end{align} 
Moreover, 
$$ 
 \| Y_{n,z}^2 \|_{p/2} \le \| ( \Delta(z,n)-\Delta(z,\infty) ) \1{\Delta(z,n) \neq \Delta(z,\infty)} \|_p^2.
 $$
  We treat the two models separately for the following intermediate calculations.

\textit{Model \ref{itm:model1}.} We define $B''_n = \{ z\in B_n : d(Q_z,\partial W_n) \le S^* \}$. Then the cardinality of $B''_n$ is at most $C S^* n^{(d-1)/d}$ for a constant $C\in\R_+$ and $\Delta(z,n)-\Delta(z,\infty) \equiv 0$ for all $z\notin B''_n$. Using Lemma~\ref{L:Moments}, we have for each $p'\ge 1$ that $\EE{ |\Delta(z,n)-\Delta(z,\infty)|^{p'} } \le C$  uniformly in $n\in\N$ and $z\in B''_n$ for a constant which only depends on $p'$. Hence, the right-hand side of \eqref{E:LILPoisson2v} is bounded above by $C (\# B''_n)^{1/2}$, which is $\cO(n^{(d-1)/(2d)})$.

\textit{Model \ref{itm:model2}.} We proceed very similarly but need to take into account that stabilization is random but at an exponential rate. Let $\gamma>1$. Define $r_n = (\log n)^{\gamma}$ and $B''_n = \{z\in B_n: d(Q_z,\partial W_n) \le r_n \}$. 
We use that $\EE{ |\Delta(z,n)|^p + |\Delta(z,\infty)|^p } \le C$ uniformly in $n\in \N$ and $z\in B_n$ by Lemma~\ref{L:Moments}. We use this upper bound if $z\in B''_n$.

If $z\notin B''_n$, we use a less conservative upper bound instead. We have
\begin{align}
		& \Delta(z,n) - \Delta(z,\infty) \nonumber \\
		\begin{split} \label{E:LILPoisson3v} 
		&= \Big\{H(\cP_n) - H(\cP''_{n,z} )  - H( \cP_n \cap B(z,r_n) ) + H( \cP''_{n,z} \cap B(z,r_n) ) \Big\}
		\end{split}\\
		\begin{split}\label{E:LILPoisson4v}
		&\quad + \Big\{	H( \cP \cap B(z,r_n)) - H(\cP''_z \cap B(z,r_n)) \\
		&\quad  - H( \cP  \cap B(z,S_z)) + H( \cP''_z  \cap B(z,S_z)) \Big\}
		\end{split}
\end{align}
because $\cP_n \cap B(z,r_n)\equiv \cP \cap B(z,r_n)$ and similarly $\cP''_{n,z}\cap B(z,r_n) \equiv \cP''_{z}\cap B(z,r_n)$. 

Using the polynomial boundedness of the functional $H$ in \ref{itm:model2} together with the representation in \eqref{E:LILPoisson3v} and \eqref{E:LILPoisson4v}, one sees that for $z\notin B''_n$
\begin{align*}
	| \Delta(z,n) - \Delta(z,\infty) | &\le C \Big( n^\beta + \#(\cP \cap W_n)^\beta + \#(\cP''_z \cap W_n)^\beta \\
	&\qquad\quad + r_n^\beta + \#( \cP \cap B(z,r_n))^\beta  + \#( \cP''_z \cap B(z,r_n))^\beta  \\
	&\qquad\quad + S_z^\beta + \#( \cP \cap B(z,S_z))^\beta + \#( \cP''_z \cap B(z,S_z))^\beta \Big)
\end{align*}
for a certain $\beta\in\R_+$. Clearly, $\E{ \# (\cP \cap B(0,w))^\beta }$ is of order $w^{\beta d}$.

Note that by the assumption of the exponential stabilization of the add-one cost function (see \eqref{E:GeneralStabilization} and \eqref{C:ExpStabilization}), it is an immediate consequence that also $\p(S_z>r_n)$ decays as $e^{-c r_n}$. So, all moments of $S_z$ exists. Moreover,
\begin{align*}
	\E{ \#( \cP \cap B(z,S_z))^\beta } &\le \sum_{n=1}^\infty \E{ \#( \cP \cap B(z,n))^\beta \1{ S_z \in (n-1,n] } } \\
	&\le \sum_{n=1}^\infty \E{ \#( \cP \cap B(z,n))^{2 \beta} }^{1/2}  \p( S_z > n-1 )^{1/2} = C_\beta < \infty.
\end{align*}
All in all, for each $\beta'>0$, there is an $\alpha>0$ such that $\EE{|\Delta(z,n)-\Delta(z,\infty)|^{\beta'}}$ is of order $n^{\alpha}$ uniformly in $z\notin B''_n$.
Consequently, using that $\p(S_z>r_n)$ is of order $n^{-c (\log n)^{\gamma-1}}$, we see with the H{\"o}lder inequality that 
$$
	\sum_{z\notin B''_n} \| \Delta(z,n)-\Delta(z,\infty) \|_p^2 = o( \# B''_n ).
$$
Moreover, by the above considerations $\sum_{z\in B''_n} \| \Delta(z,n)-\Delta(z,\infty) \|_p^2 = \cO( \# B''_n )$. So, the right-hand side of \eqref{E:LILPoisson2v} is of order $\cO( (\log n)^{\gamma/2} n^{ (d-1)/(2d) } )$.

This puts us in position to derive the maximal inequality for both models. We have
\begin{align*}
		&\E{ \max_{k\le n} \left | \sum_{z\in B_k} \E{ \Delta(z,k) - \Delta(z,\infty) | \cF_z } \right |^p }^{1/p} \nonumber \\
		&\le n^{1/p} \max_{k\le n}  \E{ \left | \sum_{z\in B_k} \E{ \Delta(z,k) - \Delta(z,\infty) | \cF_z } \right |^p }^{1/p} \le C (\log n)^{\gamma/2} n^{1/p +  (d-1)/(2d)  }.
\end{align*}
Since $\gamma/2 > 1/2$, an application of Lemma~\ref{L:LILTool} yields 
\[
	  \Big | \sum_{z\in B_n} \E{ \Delta(z,n) - \Delta(z,\infty) | \cF_z } \Big |  = \cO_{a.s.} ((\log n)^{3/2+\delta} \ n^{1/p + (d-1)/(2d) } )
	\]
for each $\delta>0$ and proves the claim as the last result it is true for all $\delta>0$.
\end{proof}

We can now prove LIL for the Poisson process.
\begin{proof}[Proof of Theorem~\ref{T:LILPoisson}]
We use the following decomposition using martingale differences
\begin{align}\begin{split}\label{E:LILPoisson1}
	H(\cP_n) - \E{ H(\cP_n)} &=   \sum_{z\in B_n }  \E{ H(\cP_n ) - H(\cP''_{n,z})  | \cF_z } = \sum_{z\in B_n } \E{ \Delta(z, n)  | \cF_z } \\
	&= \sum_{z\in B_n } \E{ \Delta(z,\infty)  | \cF_z }  + \sum_{z\in B_n } \E{ \Delta(z,n) - \Delta(z,\infty) | \cF_z }.
\end{split}\end{align}
The remainder in \eqref{E:LILPoisson1} is of order $o_{a.s.}( \sqrt{n} )$ in both models, see Proposition~\ref{P:OrderRemainderPoisson}.

The leading term involving the $F_z = \E{ \Delta(z,\infty)  | \cF_z } $, $z\in B_n$, satisfies Theorem~\ref{T:GeneralLIL}. We explain in Lemma~\ref{L:TauSigma} that $\tau^2$ is positive, hence, also $\sigma^2$ is positive. This completes the proof.
\end{proof}

\subsection{The LIL for the binomial process}
\begin{proposition}\label{P:LILBinomialError}
Assume model~\ref{itm:model1} or \ref{itm:model2}. Then for all $\delta>0$ 
\begin{align}
			&   H( \cU_{n,n} ) - H( \cU_{n,N_n} ) - \alpha (n- N_n)  = o_{a.s.}(n^{ 1/4+1/p } (\log n)^{1+\delta} ).   \label{E:LILBinomialError1}
\end{align}
\end{proposition}
\begin{proof}
We distinguish the cases $|N_n-n|\ge n^{1/2+1/p}$ and $|N_n - n| < n^{1/2+1/p}$. An application of the Borel-Cantelli Lemma reveals that $\{ |N_n - n| \ge n^{1/2+1/p} \}$ is empty for almost all $n\in\N$ with probability 1, for each $p>0$. In particular, 
\[
	\limsup_{n\to\infty} h(n) \Big| H( \cU_{n,n} ) - H( \cU_{n,N_n} ) - \alpha (n- N_n) \Big| \1{ |N_n - n| \ge n^{1/2+1/p} } = 0 \quad a.s.
\]
for any sequence $(h(n))_n \subset \R_+$.

Now, we show the claim on the set $\{ |N_n - n| < n^{1/2+1/p} \}$ in two steps. In the first step, we compute the $p$th moment of \eqref{E:LILBinomialError1}, when conditioned on $N_n$. In the second step, we conclude.

\textit{Step 1.} Using the requirement that $p\in\N_+$ is even, we show
\begin{align}\begin{split}\label{E:LILBinomialErrorDiff1}
	&\E{  \Big| H( \cU_{n,n} ) - H( \cU_{n,N_n} ) - \alpha (n- N_n) \Big|^p \Big| N_n } \1{ |N_n-n| \le n^{1/2+1/p} } \\
	& \le C |N_n - n|^{p/2} \quad a.s.
\end{split}\end{align}
for a constant $C\in\R_+$, which does neither depend on $N_n$ nor on $n$. We use the decomposition
\begin{align}\label{E:LILBinomialErrorDiff1b}
		| H( \cU_{n,n} ) - H( \cU_{n,N_n} ) - \alpha (n- N_n) | &= \Big| \sum_{m \in I_n }  (H( \cU_{n,m+1} ) - H( \cU_{n,m} ) - \alpha ) \Big|,
\end{align}
where $I_n$ is given in \eqref{E:LILBinomialError1b}. Let $n\in\N$ be arbitrary but fixed. Define $Y_m \coloneqq \ \} (\cU_{n,m+1} \setminus \cU_{n,m}) \{$ as the unique point in $\cU_{n,m+1}$ that is not contained in $\cU_{n,m}$. 

Computing the left-hand side of \eqref{E:LILBinomialErrorDiff1} with the equality in \eqref{E:LILBinomialErrorDiff1b}, amounts to calculate
\begin{align}
		&\sum_{ (m_1,\ldots,m_p) \in I_n^p } \E{ (H( \cU_{n,m_1+1} ) - H( \cU_{n,m_1} ) - \alpha ) \cdot \ldots \cdot (H( \cU_{n,m_p+1} ) - H( \cU_{n,m_p} ) - \alpha ) } \nonumber \\
		\begin{split}\label{E:LILBinomialErrorDiff2}
		&= \sum_{ (m_1,\ldots,m_p) \in I_n^p} \mathbb{E}\Big[ \mathbb{E}\Big[ (H( \cU_{n,m_1+1} ) - H( \cU_{n,m_1} ) - \alpha ) \cdot \ldots \\
		&\qquad\qquad\qquad\qquad \cdot (H( \cU_{n,m_p+1} ) - H( \cU_{n,m_p} ) - \alpha ) \ \Big| Y_{m_1}, \ldots, Y_{m_p} \Big] \Big].
		\end{split}
\end{align}
Let $(m_1,\ldots,m_p)$ be a generic tuple and consider a factor $(H( \cU_{n,m_k+1} ) - H( \cU_{n,m_k} ) - \alpha )$ which occurs with a multiplicity of $a_k\in \{1,\ldots,p\}$ in the product.
The details depend now on the selected model.

\textit{Model~\ref{itm:model1}.} 
Denote $Z_m = H( \cU_{n,m+1} ) - H( \cU_{n,m} ) - \alpha $. Let $\cP_{n,m}$ be the homogeneous Poisson process on $\R^d$, which consists of the Poissonized binomial process $\cU_{n,m}$ on $W_n$ and an independent homogeneous Poisson process $\cP^*$ on $\R^d\setminus W_n$. Note that $\cP_{n,m}$ is independent of $Y_{m}$. Then,
\begin{align*}
	Z_m &= \{ H( \cU_{n,m+1} ) - H( \cU_{n,m} ) -  H( \cP_{n,m}\cup \{Y_m\} ) + H( \cP_{n,m} ) \} \\
	&\quad + \{ H( \cP_{n,m}\cup \{Y_m\} ) - H( \cP_{n,m} ) -  \alpha \} =: Z'_m + Z^{\dagger}_m
\end{align*}
and $\EE{ Z^\dagger_m | Y_m } = 0$. Consider a generic tuple $(m_1,\ldots,m_p)$ in the sum in \eqref{E:LILBinomialErrorDiff2} and replace each factor $Z_{m_k}$ by $Z'_{m_k}+Z^\dagger_{m_k}$.

If $A_k = \{ d(Y_{m_k}, \cup_{j: j\neq k} Y_{m_j} ) > 2S^*\}$ occurs, then $Z^\dagger_{m_k}$ is independent of all other factors conditionally on the points $Y_{m_1},\ldots,Y_{m_p}$. And if additionally the multiplicity $a_k=1$, then the expectation of this factor is close to zero. Otherwise, if $A_k$ does not occur, there is a certain dependence between this and other factors, however, $1 - \p( A_k) \le c n^{-1}$.

These ideas will then lead to the claim. We introduce now some notation and write 
\[
	L = \{k \in  \{1,\ldots,p\} \ | \ m_i \neq m_k: \ \forall i < k, i\in\{1,\ldots,p\} \}
	\]
for all indices of the generic tuple which occur and $L^* = \{k\in L: a_k = 1 \}$ for the indices which occur exactly once. (So, $L$ and $L^*$ both depend on the tuple $(m_1,\ldots,m_p)$.) 

Consider the expectation in \eqref{E:LILBinomialErrorDiff2} for a generic tuple $(m_1,\ldots,m_p)$ now in terms of the multiples $a_k$. We have
\begin{align}
		&\E{ \E{ \prod_{k\in L } Z_{m_k} ^{a_k} \ \Bigg| Y_{m_k}: k\in L } } \nonumber \\
		&= \mathbb{E}\Bigg[  \prod_{k\in L^* } ( \1{ A_k} + \1{A_k^c} ) \ \mathbb{E}\Bigg[  \prod_{k\in L^* }  Z_{m_k}   \prod_{k\in L \setminus L^* }  Z_{m_k} ^{a_k} \    \Bigg| Y_{m_k}: k\in L \Bigg]  \Bigg] \nonumber \\		
		&= \sum_{ w\in \{0,1\}^{\# L^*} } \mathbb{E}\Bigg[  \prod_{k\in L^* } \1{ A_k}^{w_k} \1{A_k^c}^{1-w_k} \ \mathbb{E}\Bigg[  \prod_{k\in L^* }  Z_{m_k} \prod_{k\in L \setminus L^* }  Z_{m_k}^{a_k} \    \Bigg| Y_{m_k}: k\in L \Bigg]  \Bigg] \nonumber \\		
		\begin{split}
		&= \sum_{ w\in \{0,1\}^{\# L^*} } \mathbb{E}\Biggl[ \prod_{\substack{k\in L^*:\\ w_k=0} } \1{A_k^c}  \mathbb{E}\Biggl[ \prod_{\substack{k\in L^*:\\ w_k=1} } \Big\{ \1{ A_k} \1{E_{m_k}}  Z'_{m_k} \Big\} \\
		&\qquad\qquad\qquad\qquad\qquad  \prod_{\substack{k\in L^*:\\ w_k=0} }  Z_{m_k}   \prod_{k\in L \setminus L^* }  Z_{m_k}^{a_k} \    \Bigg| Y_{m_k}: k\in L \Biggl] \Biggl], \label{E:LILBinomialErrorDiff3}
		\end{split}
\end{align}
where $E_{m} =  \{ \cU_{n,m} \cap B(Y_{m},S^*)  \neq \cP_{n,m} \cap B(Y_{m},S^*) \}$ and where we use that $Z'_{m_k}$ can only differ from 0 if $E_{m_k}$ occurs. We study the deviations of $Z'_{m_k}$ from zero:\begin{align}
			\p( E_{m_k} | Y_{m_k} ) &\le \mathbb{E}\Big[ \sum_{Y\in \cP_{n,m_k} \setminus \cU_{n,m_k} } \1{ Y \in B(Y_{m_k},S^*) } \nonumber \\
			&\quad\qquad\qquad + \sum_{Y\in \cU_{n,m_k} \setminus  \cP_{n,m_k}} \1{ Y \in B(Y_{m_k},S^*) }  | Y_{m_k} \Big]\nonumber\\
			&\le \frac{C}{n} \ \E{ |\cP_{n,m_k}(W_n) - m_k | + \cP_{n,m_k}(B(Y_{m_k},S^*) \cap (\R^d\setminus W_n) ) } \nonumber\\
			&\le \frac{C}{n} \ ( n^{1/2} + n^{1/2+1/p} + 1 ) = \cO(n^{-(1/2-1/p)}). \label{E:LILBinomialErrorDiff4a}
\end{align}
Furthermore, using the condition that $|N_n-n| \le n^{1/2+1/p}$ and Lemma~\ref{L:Moments}, we find for each $\wt p\in\N$
\begin{align}
		 \E{  \big|Z_{m_k}  \big|^{\wt p} \big| Y_{m_k} } \vee  \E{  \big|Z_{m_k}  \big|^{\wt p} \big| Y_{m_k} : k \in L} &\le  C, \label{E:LILBinomialErrorDiff5}
\end{align}
uniformly in $m_k \in I_n$ and $n\in\N$, clearly, the similar upper bounds are valid for $Z'_{m_k}$. 

 Then \eqref{E:LILBinomialErrorDiff3} is bounded above by
 \begin{align}
 	&C\ \cdot \sum_{ w\in \{0,1\}^{\# L^*} } \mathbb{E}\Big[  \prod_{\substack{k\in L^*:\\ w_k=0} } \1{A_k^c}  \prod_{\substack{k\in L^*:\\ w_k=1} }  \Big\{ \1{ A_k} \p( E_{m_k} |Y_{m_k})^{1/r} \Big\} \Big] \nonumber \\
	&\le C \ \cdot \sum_{ w\in \{0,1\}^{\# L^*} } \p( \cap_{k\in L^*: w_k = 0 } A_k^c ) \ n^{-(1/2-1/p) \|w\|_1 / r},
	\label{E:LILBinomialErrorDiff4}
 \end{align}
where $\|\cdot\|_1$ is the $\ell^1$-norm and where we apply the H{\"o}lder inequality with $r\in (1, (2p-4)/(p+2) )$ together with the result from \eqref{E:LILBinomialErrorDiff4a} and \eqref{E:LILBinomialErrorDiff5} (for $\wt p = p r / (r-1)$). It remains to derive an upper bound for the probability in \eqref{E:LILBinomialErrorDiff4}. To this end, set $b = \#L^* - \|w\|_1$. Let $\{k\in L^*: w_k = 0\}$ be given by $\{k_1,\ldots,k_b\}$. Then
\begin{align}
			\p( \cap_{k\in L^*: w_k = 0 } A_k^c ) &= \p( d( Y_{m_k}, \cup_{j: j\neq k} Y_{m_j} ) \le 2S^* \text{ for all } k\in L^* \text{ with } w_k=0 ) \nonumber \\ 
			&\le \sum_{j_1,\ldots,j_b = 1}^p \p( d( Y_{m_{k_1}}, Y_{m_{j_1}}) \le 2 S^* , \ldots, d( Y_{m_{k_b}}, Y_{m_{j_b}}) \le 2 S^* ) \label{E:LILBinomialErrorDiff5a}\\			
			&\le C n ^{- \ceil{(\# L^* - \|w\|_1)/2 } }. \label{E:LILBinomialErrorDiff5a2}
\end{align}
The last inequality can be derived as follows: Since the indices $m_{k_1},\ldots,m_{k_b}$ are all pairwise different, there are at least $\ceil{ b/2 }$ different sets of the type $\{Y_{m_{k_i}}, Y_{m_{j_i}}\}$ within each probability in \eqref{E:LILBinomialErrorDiff5a}. This corresponds to at least $\ceil{b/2}$ independent random variables and the claim follows.

Consequently, using the result from \eqref{E:LILBinomialErrorDiff4} and \eqref{E:LILBinomialErrorDiff5a2}, \eqref{E:LILBinomialErrorDiff3} is of order
\begin{align*}
		&\sum_{ w\in \{0,1\}^{\# L^*} }  n^{ -\ceil{(\# L^* - \|w\|_1)/2} } \ n^{-(1/2-1/p)  \| w \|_1 / r } \\
		&\le n^{-\# L^*/2} (1+n^{1/2 - 1/(2r) + 1/(pr) })^{\# L^*} \le C n^{ - \# L^* (1/2-1/p)/r }.
\end{align*}
We come to the conclusion for the sums in \eqref{E:LILBinomialErrorDiff2}. Evidently, the cardinality of $ L^*$, which counts the tuples in $(m_1,\ldots,m_p)$ that occur exactly once, can range from 0 to $p$. So, it remains to determine the number of combinations for which $\# L^* = j$ for each $j\in\{1,\ldots,p\}$. This can be done with an elementary combinatorial argument: The number of combinations which determine the indices that occur exactly once is at most $(\# I_n)^j$. The number of free positions in $I_n$ is $p-j$, however, each position has to be occupied by two indices, hence, it remain at most $(\# I_n)^{ \floor{(p-j)/2} }$ combinations for the indices with a multiplicity of at least 2. This means the number of combinations for which $\# L^* = j$ is at most $(\# I_n)^{ j +  \floor{(p-j)/2} } = (\# I_n)^{ \floor{j/2} + p/2 }$ because $p$ is even. Consequently,
\[
		\# \{ (m_1,\ldots,m_p) \in I_n^p: L^*(m_1,\ldots,m_p) = j \} \le (\# I_n )^{ p/2 + \floor{ j /2}  }.
\]
Consequently, since $r\in (1, (2p-4)/(p+2) )$ and $(\# I_n ) \le n^{1/2+1/p}$, \eqref{E:LILBinomialErrorDiff2} is at most (times a suitable constant which is independent of $n$)
\begin{align}\begin{split}\label{E:LILBinomialErrorDiff5b}
	\sum_{(m_1,\ldots,m_p)\in I_n^p } \frac{1}{n^{ \# L^* (1/2-1/p)/r } } &\le \sum_{j=0}^{p}   \frac{(\# I_n )^{ p/2 +  \floor{ j/2 } } }{n^{j (1/2-1/p)/r }} \\
	&\le (\# I_n)^{p/2} \sum_{j=0}^{p}  \frac{n^{ (1/2+1/p ) \floor{ j/2 }} } {n^{  j (1/2-1/p)/r} } \le (p+1) (\# I_n)^{p/2}
	\end{split}
\end{align}
because of the choice of $r$. This shows \eqref{E:LILBinomialErrorDiff1} for the model \ref{itm:model1}. 

\textit{Model \ref{itm:model2}.} Let $n\in\N$ be arbitrary but fixed and let $r\in (1, (2p-4)/(p+2) )$ and $r_n \coloneqq n^\gamma$ for $\gamma \in (0, ( (2p-4) - (p+2)r )/(4dp) )$ (see also \eqref{E:ChoiceGamma} below). We keep the notation introduced in the calculations for model~\ref{itm:model1} and decompose the right-hand side of \eqref{E:LILBinomialErrorDiff1b} in two terms as follows
\begin{align}
			&\Big| \sum_{m\in I_n} \Big[ H(\cU_{n,m+1}) - H(\cU_{n,m}) \Big] - \Big[ H(\cU_{n,m+1}\cap B(Y_m,r_n) - H(\cU_{n,m} \cap B(Y_m,r_n)) \Big] \Big| \nonumber\\
			&=: \Big| \sum_{m\in I_n} \wt Z_m \Big|, \label{E:LILBinomialErrorDiff6} \\
			&\Big| \sum_{m\in I_n} \Big[ H(\cU_{n,m+1}\cap B(Y_m,r_n) - H(\cU_{n,m} \cap B(Y_m,r_n)) \Big] - \alpha \Big|. \label{E:LILBinomialErrorDiff7}
\end{align}
One can follow the calculations for model~\ref{itm:model1} to show that conditional on the event $\{ |N_n -n| \le n^{1/2+1/p}\}$ the $p$th-moment of the term in \eqref{E:LILBinomialErrorDiff7} is of order
\begin{align*}
	\sum_{j=0}^{p}  | I_n |^{ p/2 +  \floor{ j/2 } } \Big( \frac{r_n^d}{n^{1/2-1/p}} \Big)^{j/r} \le (p+1) |I_n|^{p/2},
\end{align*}
for the above choices of $\gamma$ and $r$ (compare with \eqref{E:LILBinomialErrorDiff4a}, \eqref{E:LILBinomialErrorDiff5a}, \eqref{E:LILBinomialErrorDiff5b}, where $r_n$ corresponds to $S^*$). Note that in this case we need the uniform bounded moments condition from \eqref{C:UniformBoundedMoments}
to be satisfied for $r/(r-1)p$, in particular,
\begin{align*}
	&\E{ |H(\cU_{n,m+1}) - H(\cU_{n,m})|^{p r/(r-1)} } \le C \qquad \text{ and}\\
	&	\E{ | H(\cU_{n,m+1}\cap B(Y_m,r_n) - H(\cU_{n,m} \cap B(Y_m,r_n))  |^{p r/(r-1)} } \le C
		\end{align*}
uniformly in $n$ and $m \in [1/2 n, 3/2 n]$. 

 Consider the approximation error in \eqref{E:LILBinomialErrorDiff6} on $\{ |N_n -n| \le n^{1/2+1/p}\}$. If $\wt Z_m \neq 0$, then
\[
	E_m =  \big\{ \cU_{n,m} \cap B(Y_m,r_n) \neq  \cP_n \cap B(Y_m,r_n) \big\}	\cup 	\big\{ \wt S > r_n \big\} 
\]
occurs; here $\wt S = \wt S(-Y_m,n,(\cP_n - Y_m)( Q_0),-Y_m)$ is the radius of stabilization of the point clouds $(\cP_n - Y_m)$ and $\{0 \}$ under the functional $H$ from \eqref{E:GeneralStabilization} and \eqref{C:ExpStabilization}. Consequently,
\begin{align}
		&\E{ \Big|\sum_{m\in I_n } \wt Z_m  \Big|^p \ \Big| \ N_n } \nonumber\\
		&= \sum_{(m_1,\ldots,m_p) \in I_n^p } \E{ \wt Z_{m_1}\1{ E_{m_1} } \cdot \ldots \cdot \wt Z_{m_p}\1{ E_{m_p} } \ | \ N_n } \nonumber \\
		&\le \sum_{(m_1,\ldots,m_p) \in I_n^p } \E{ | \wt Z_{m_1} \cdot \ldots \cdot \wt Z_{m_p} |^{r/(r-1)} }^{(r-1)/r} \p( E_{m_1} \cap \ldots \cap E_{m_p}  )^{1/r}. \label{E:LILBinomialErrorDiff8}
\end{align}
We focus on the probability in \eqref{E:LILBinomialErrorDiff8}. Once more, consider a generic tuple $(m_1,\ldots,m_p)$.
Since $\p( \wt S(y,n,k,x) > r_n )$ decays exponentially in $n$ (uniformly in the parameters $y,n,k,x$), the relevant part of the probability in \eqref{E:LILBinomialErrorDiff8} is
\begin{align}\label{E:LILBinomialErrorDiff9}
		&\p(  \cU_{n,m_i} \cap B(Y_{m_i},r_n) \neq  \cP_n \cap B(Y_{m_i},r_n) \text{ for } i \in \{1,\ldots,p\}  ).
\end{align}
We partition the domain $\Omega$ suitably,
$$
	\Omega = \bigcup_{w\in\{0,1\}^p } \Big( \bigcap_{i: w_i =1 } A_i \cap \bigcap_{i: w_i =0 } A_i^c \Big),
	$$
where $A_k = \{ d( Y_{m_k}, \cup_{j: j\neq k} Y_{m_j} )  > 2 r_n \}$ for $k \in\{1,\ldots, p\}$. Let $w\in \{0,1\}^p$. Then using the properties of the Poisson process, we obtain for this $w$
\begin{align*}
		&\p\Big(	\cU_{n,m_i} \cap B(Y_{m_i},r_n) \neq \cP_n \cap B(Y_{m_i},r_n) \text{ for } i\in\{1,\ldots,p \} \text{ with } w_i =1 \\
		&\qquad \big | \ Y_{m_i}, \1{A_i} \text{ for } i \in \{1,\ldots,p \} \text{ with } w_i = 1, \cU_{n,m_k} \cap B(Y_{m_k},r_n) \text{ for } k\in\{1,\ldots,p \} \Big) \\
		&= \prod_{i: w_i=1} \p\Big(	\cU_{n,m_i} \cap B(Y_{m_i},r_n) \neq \cP_n \cap B(Y_{m_i},r_n) \ \big | \ Y_{m_i}, \1{A_i}, \cU_{n,m_i} \cap B(Y_{m_i},r_n) \Big) \\
		&\le \prod_{i: w_i=1} \frac{C r_n^d}{n} \Big( \E{ |N_n - n| } + |n-m| \Big) = \cO\Big( \Big(\frac{ r_n^d}{n^{1/2 - 1/p} } \Big)^{\|w\|_1} \Big).
\end{align*}
Thus, \eqref{E:LILBinomialErrorDiff9} equals
\begin{align*}
		& \sum_{w\in \{0,1\}^p } \p\Big( \big\{ \cU_{n,m_i} \cap B(Y_{m_i},r_n) \neq  \cP_n \cap B(Y_{m_i},r_n) \text{ for } i \in \{1,\ldots,p\} \big\} \\
		&\qquad\qquad\quad \cap \bigcap_{i: w_i = 1 } A_i \cap \bigcap_{i: w_i = 0} A_i^c \Big ) \\
		&\le \sum_{w\in \{0,1\}^p } \Big(	\frac{ C r_n^d }{n^{1/2 - 1/p} } \Big)^{\|w\|_1 } \p\Big(	\bigcap_{i: w_i = 0} A_i^c \Big ) \\
		&\le \sum_{w\in \{0,1\}^p } \Big(	\frac{ C r_n^d }{n^{1/2 - 1/p} } \Big)^{\|w\|_1 } \Big(	\frac{ C r_n^d }{n} \Big)^{ \ceil{(p - \|w\|_1)/2 } } \\
		&\le C \Big(	\frac{ r_n^d }{n^{1/2 - 1/p} } \Big)^{p},
\end{align*}
for the derivation of the second but last inequality see  \eqref{E:LILBinomialErrorDiff5a}.		

So the $p$th moment of \eqref{E:LILBinomialErrorDiff6} on $\{ |N_n -n| \le n^{1/2 + 1/p}\}$ is bounded above by (up to a multiplicative constant)
\begin{align}\begin{split}\label{E:ChoiceGamma}
		&\Big(	\frac{ r_n^d }{n^{1/2 - 1/p} } \Big)^{p/r } |N_n - n|^p \1{ |N_n-n|\le n^{1/2 + 1/p} } \\ &\le  \Big(	\frac{ r_n^d }{n^{1/2 - 1/p} } \Big)^{p/r } n^{(1/2+1/p) p/2} \ |N_n - n|^{p/2}.
\end{split}\end{align}
If we choose $r\in (1, (2p-4)/(p+2) )$ and $\gamma \in (0, ( (2p-4) - (p+2)r )/(4dp) )$, then the left-hand side of \eqref{E:ChoiceGamma} is $o( |N_n - n|^{p/2} )$. In particular, \eqref{E:LILBinomialErrorDiff1} is also satisfied in model~\ref{itm:model2}.

\textit{Step 2.} Using Jensen's inequality and properties of the maximum, we find with the result of the first step
\begin{align*}
		& \E{ \max_{k \in \{1,\ldots,n\} } \Big| H( \cU_{k,k} ) - H( \cU_{k,N_k} ) - \alpha (k- N_k) \Big|  \1{ |N_k-k| \le k^{1/2 + 1/p} } } \\
		&\le n^{1/p} \max_{k \in \{1,\ldots,n\} } \mathbb{E}\Big[ \Big| H( \cU_{k,k} ) - H( \cU_{k,N_k} ) - \alpha (k- N_k) \Big|^p \\
		&\quad\qquad\qquad\qquad\qquad \cdot \1{ |N_k-k| \le k^{1/2 + 1/p} } \Big]^{1/p}   \\
		&= n^{1/p} \max_{k \in \{1,\ldots,n\} } \mathbb{E} \Big[  \E{ \Big| H( \cU_{k,k} ) - H( \cU_{k,N_k} ) - \alpha (k- N_k) \Big|^p \  \Big | N_k } \\
		&\quad\qquad\qquad\qquad\qquad \cdot \1{ |N_k-k| \le k^{1/2 + 1/p} } \Big]^{1/p}   \\
		&\le C n^{1/p} \max_{k \in \{1,\ldots,n\} } \E{ |N_k - k|^{p/2} }^{1/p} \\
		&\le C n^{1/4+ 1/p},
		\end{align*}
where the last inequality follows because the $\ell$th centered moment of a Poisson random variable with parameter $\lambda$ is bounded above by $C_\ell \lambda^{\floor{\ell/2}}$ for some constant $C_\ell \in\R_+$ which only depends on $\ell$. The claim follows now from Lemma~\ref{L:LILTool}, using the additional factor $(\log n)^{1+\delta}$, for some $\delta >0$.
\end{proof}

\begin{proposition}\label{P:LILBinomialError2}
\hspace{-12pt} Assume model~\ref{itm:model1} or \ref{itm:model2}. Then $ \E{ H(\cU_{n,N_n} ) - H(\cU_{n,n}) } = \cO(n^{1/4})$.
\end{proposition}
\begin{proof}
We will use the decomposition
\[
	H(\cU_{n,N_n} ) - H(\cU_{n,n}) = \sum_{m=n}^{N_n-1} H(\cU_{n,m+1} ) - H(\cU_{n,m}) - \sum_{m=N_n}^{n-1} H(\cU_{n,m+1} ) - H(\cU_{n,m}),
\]
with the convention that the first or the second sum is zero depending on whether $N_n < n$ or $N_n>n$. Moreover, we have $0= \EE{ \alpha( N_n-n) } = \EE{ \alpha( N_n-n) \1{ N_n > n} - \alpha(n - N_n) \1{N_n < n} }$. 
Using another time the definition of the set $I_n$ from \eqref{E:LILBinomialError1b}, we obtain
\begin{align}
 		&\Big |\E{ H(\cU_{n,N_n} ) - H(\cU_{n,n}) } \Big | \nonumber\\
		&\le  \E{ \Big | \sum_{m\in I_n} \Big( H(\cU_{n,m+1} ) - H(\cU_{n,m}) - \alpha \Big) \Big | } \nonumber \\
		&\le  \E{ \Big | \sum_{m\in I_n} \Big( H(\cU_{n,m+1} ) - H(\cU_{n,m}) - \alpha  \Big) \Big | \1{|N_n - n | \le n^{1/2 + 1/p} } } \label{E:LILBinomialErrorExp1a} \\
			&\quad +  \E{ \Big | \sum_{m\in I_n} \Big( H(\cU_{n,m+1} ) - H(\cU_{n,m}) - \alpha \Big) \Big | \1{|N_n - n | > n^{1/2 + 1/p} } }. \label{E:LILBinomialErrorExp1b}
\end{align}
First consider \eqref{E:LILBinomialErrorExp1a}. We deduce from \eqref{E:LILBinomialErrorDiff1} that for both models
\begin{align*}
			&\E{  \Big | \sum_{m\in I_n} \Big( H(\cU_{n,m+1} ) - H(\cU_{n,m}) - \alpha\Big)  \Big|^p  \ \Big| \ N_n } \\
			 &\qquad\qquad \1{ |N_n-n| \le n^{1/2 + 1/p}}  \le C |N_n - n|^{p/2} \quad a.s.
\end{align*}
Hence, the expectation in \eqref{E:LILBinomialErrorExp1a} is of order $\cO(n^{1/4})$ in both models~\ref{itm:model1} and \ref{itm:model2}.

Second, we consider \eqref{E:LILBinomialErrorExp1b} which is less than
\begin{align}\begin{split}\label{E:LILBinomialErrorExp1c}
		&\E{ \sum_{m\in I_n} | H(\cU_{n,m+1} ) - H(\cU_{n,m}) | \  \1{|N_n - n | > n^{1/2 + 1/p} } } \\
		&\quad + \alpha \ \E{ (N_n - n)^2 }^{1/2} \p(|N_n - n | > n^{1/2 + 1/p}	)^{1/2}
\end{split}\end{align} 
The second term in \eqref{E:LILBinomialErrorExp1c} decays exponentially and is therefore of order $o(n^{1/4})$. Regarding the first term, we need to argue differently for each model.

We begin with model \ref{itm:model1}. Let $n\in\N$ be arbitrary but fixed. Once more, write $Y_m = \} (\cU_{n,m+1} \setminus \cU_{n,m}) \{$. Denote $M_{n,m}$ the number of points of $\cU_{n,m}$ in $B(Y_m, S^*)$. Then $M_{n,m}$ follows a binomial distribution with length $m$ and a probability of order $(S^*)^d / n$. Thus, 
\begin{align*}
		\sum_{m\in I_n} \E{ | H(\cU_{n,m+1} ) - H(\cU_{n,m}) | \ | N_n } &\le \sum_{m\in I_n} \E{ 2 C^* e^{q M_{n,m}} } \\
		&\quad \le c' |N_n-n| \big(	e^{c N_n / n}  + 1 \big).
\end{align*}
for certain constants $c,c' \in\R_+$. Next, $\EE{\exp( c N_n/n )} = \exp( (e^{c  / n} -1) n ) \to \exp( c )$ as $n\to \infty$. This shows that also the first term in \eqref{E:LILBinomialErrorExp1c} decays at an exponential rate in model~\ref{itm:model1}.

Finally, consider \ref{itm:model2}. $|H(\cU_{n,m+1} )- H(\cU_{n,m} )|$ is of order $\cO_{a.s.}( (N_n + n)^\nu)$, for some $\nu\in\R_+$. As the probability $\p( |N_n - n| > n^{1/2 + 1/p} )$ decays exponentially, an application of the H{\" o}lder inequality shows that the first term in \eqref{E:LILBinomialErrorExp1c} is also of order $o(n^{1/4})$ under the assumptions of model~\ref{itm:model2}.
\end{proof}

\begin{lemma}\label{L:PoissonPointsRem}
For all $\delta>0$,
$$
	\sum_{z\in B_n} \E{ \Delta'(z,n) - \Delta'(z,\infty)  | \cF_{z} }  = o_{a.s.}(n^{(1/2-1/(2d))} (\log n)^{2+\delta}).
$$
\end{lemma}
\begin{proof}
We have $ \E{ \Delta'(z,n) - \Delta'(z,\infty)  | \cF_{z} } =  -\cP(Q_z  \setminus B_n) + \E{ \cP(Q_z  \setminus B_n)} $, which is zero if $Q_z \subset B_n$. Write $B''_n$ for those $z\in B_n$ such that $Q_z$ is not a subset of $B_n$. $\# B''_n$ is of order $\cO(n^{(d-1)/d})$. Define  $W_{n,z} = \cP(Q_z  \setminus B_n)$. Consider the Laplace transform
\begin{align}\label{E:PoissonPointsRem1}
  		\E{\exp( \gamma \sum_{z\in B_n} W_{n,z} - \E{W_{n,z}} ) } &= \prod_{z\in B''_n} \E{ \exp(\gamma W_{n,z}) } \exp( - \gamma \E{ W_{n,z} } ).
\end{align}
Since $W_{n,z}$ is Poisson distributed with parameter $|Q_z \setminus B_n| = \E{ W_{n,z} }$, we have for all $\gamma\in\R_+$
\[
	\E{ \exp(\gamma W_{n,z}) }  = \exp\big( \E{ W_{n,z} } (e^\gamma - 1) \big).
\]
So, the right-hand side of \eqref{E:PoissonPointsRem1} equals
$
		\exp(\sum_{z\in B''_n} \E{ W_{n,z}}( e^\gamma - 1 -\gamma)	) \le \exp( c \gamma^2 \# B''_n ).
$
Consider
\begin{align}
		&\E{ \max_{1\le k\le n}  \Big|  \sum_{z\in B_k} W_{k,z} - \E{W_{k,z}} 		\Big| } \label{E:PoissonPointsRem2} \\
		 &\le n^a \log \E{ \max_{1\le k\le n} \exp\Big( \frac{1}{n^a} \Big|  \sum_{z\in B_k} W_{k,z} - \E{W_{k,z}} 		\Big| \Big) } \nonumber \\
		&\le n^a \log n + n^a  \log \Bigg\{ \max_{1\le k\le n} \E{  \exp\Big( \frac{1}{n^a}  \sum_{z\in B_k} W_{k,z} - \E{W_{k,z}} \Big) } \Bigg\} \nonumber \\
		&\quad +  n^a  \log \Bigg\{ \max_{1\le k\le n} \E{  \exp\Big( - \frac{1}{n^a}  \sum_{z\in B_k} W_{k,z} - \E{W_{k,z}} \Big) } \Bigg\} \nonumber \\
		&\le  n^a \log n + c n^{a-2a} n^{(d-1)/d}.
\end{align}
Hence, if $d>1$ or $d=1$, the choice  $a=(d-1)/(2d)$ equalizes the rate in both terms, which is then $n^{(d-1)/(2d)} \log n$. An application of Lemma~\ref{L:LILTool} yields the claim.
\end{proof}

\begin{proof}[Proof of Theorem~\ref{T:LILBinomial}]
We apply the following fundamental decomposition.
\begin{align}
		H( \cU_{n,n} ) - \E{ H(\cU_{n,n}) } &= H( \cU_{n,N_n} ) - \E{ H( \cU_{n,N_n} ) } - \alpha(N_n - n) \label{E:LILBinomial1} \\
		&\quad +   \Big\{ H( \cU_{n,n} ) - H( \cU_{n,N_n} ) - \alpha (n- N_n) \Big\} \label{E:LILBinomial2} \\
		&\quad + \Big\{ \E{ H(\cU_{n,N_n} ) - H(\cU_{n,n}) } \Big\}. \label{E:LILBinomial3}
\end{align}
The term in \eqref{E:LILBinomial2} is $o_{a.s.}(\sqrt{n})$, see Proposition~\ref{P:LILBinomialError}. The term in \eqref{E:LILBinomial3} is $\cO(n^{1/4})$, see Proposition~\ref{P:LILBinomialError2}. The main term in \eqref{E:LILBinomial1} equals
\begin{align}
	&H( \cU_{n,N_n} ) - \E{ H( \cU_{n,N_n} ) } - \alpha(N_n - n)  = H(\cP_n) - \E{ H(\cP_n) } - \alpha( \# \cP_n - \E{ \# \cP_n} ) \nonumber \\
	&  =  \sum_{ z\in B_n} \E{ H(\cP_n) - H(\cP''_{n,z} ) - \alpha \ (\# \cP_n  - \# \cP''_{n,z} )  | \cF_{z} } \nonumber \\
	&=  \sum_{z\in B_n} \E{ \Delta(z,n) - \alpha \ \Delta'(z,n)  | \cF_{z} } \nonumber  \\
	&=  \sum_{z\in B_n} \E{ \Delta(z,\infty) - \alpha \ \Delta'(z,\infty)  | \cF_{z} } \label{E:LILBinomial4} \\
	& \quad + \Big\{ \sum_{z\in B_n} \E{ \Delta(z,n) - \Delta(z,\infty)  | \cF_{z} } \Big\} - \alpha \ \Big\{ \sum_{z\in B_n} \E{ \Delta'(z,n) - \Delta'(z,\infty)  | \cF_{z} } \Big\}. \label{E:LILBinomial5}
\end{align}
We show that the term in \eqref{E:LILBinomial4} satisfies the LIL, whereas the two terms in \eqref{E:LILBinomial5} are negligible. We begin with the remainder terms in \eqref{E:LILBinomial5}. The first term in \eqref{E:LILBinomial5} is of order $o_{a.s.}(\sqrt{n})$, this is demonstrated in Proposition~\ref{P:OrderRemainderPoisson}. Regarding the second term, we infer from Lemma~\ref{L:PoissonPointsRem} that it is $o_{a.s.}(\sqrt{n})$.

Finally, consider \eqref{E:LILBinomial4}. The LIL follows along the same lines as in the proof of Theorem~\ref{T:LILPoisson}, see also Remark~\ref{R:Binomial}; we omit the details. Hence,
\begin{align*}
    		& \limsup_{n\to\infty} \pm \frac{1}{\sqrt{ 2 n \log\log n} } \sum_{z\in B_n } \E{ \Delta(z,\infty) - \alpha \Delta'(z,\infty)  | \cF_z } \\
		&=  \EE{ \E{ \Delta(0,\infty) - \alpha \Delta'(0,\infty)  | \cF_0 }^2 }^{1/2} \quad a.s.
\end{align*}
We infer from Lemma~\ref{L:TauSigma} that  the $a.s.$-limit on the right-hand side equals $\tau^2=\sigma^2 - \alpha^2$ and is positive.
\end{proof}

\begin{lemma}\label{L:TauSigma}
Assume model~\ref{itm:model1} or model~\ref{itm:model2}. Then 
$$
	\tau^2 =  \EE{ \E{ \Delta(0,\infty) - \alpha \Delta'(0,\infty)  | \cF_0 }^2 }>0.
$$
\end{lemma}
\begin{proof}
If $\fD_0(\cP,\infty)$ is nondegenerate, $\tau^2$ is positive, see \cite{penrose2001central} (their proof applies to both models). It remains to verify the equality. We use the representation from \eqref{E:LILBinomial1} to \eqref{E:LILBinomial3} and set
\begin{align*}
		X_n &:= H( \cU_{n,n} ) - \E{ H(\cU_{n,n}) } \\
		Y_n &:= H( \cU_{n,N_n} ) - \E{ H( \cU_{n,N_n} ) } \\
		Z_n &:= \alpha(N_n - n).
\end{align*}
Then $(n^{-1/2} Y_n)_n$ converges to $\cN(0,\sigma^2)$ in distribution under model~\ref{itm:model1} and \ref{itm:model2} by Proposition~\ref{P:NormalApproximation} and Proposition~\ref{P:OrderRemainderPoisson}. Furthermore, $(n^{-1/2} Z_n)_n$ converges in distribution to $\cN(0,\alpha^2)$. Using the independence of $X_n$ and $Z_n$, we can infer with characteristic functions and the above approximation results that $n^{-1/2} X_n$ converges to $\cN(0,\sigma^2-\alpha^2)$, for a sketch see the proof of Theorem 2.1 in \cite{penrose2001central}. It remains to show the equality $\tau^2 = \EE{ \E{ \Delta(0,\infty) - \alpha \Delta'(0,\infty)  | \cF_0 }^2 }$.

To this end, we first show
\begin{align}
		\tau^2 \overset{\text{def}}{=} \sigma^2 - \alpha^2 &= \lim_{n\to\infty} n^{-1} \V{ H(\cU_{n,n}) } \label{E:TauSigma1} \\
		\begin{split}\label{E:TauSigma2}
		&= \lim_{n\to\infty} \big\{ n^{-1} \V{ H(\cU_{n,N_n} )} + \alpha^2 n^{-1} \V{ N_n } \\
		&\qquad\quad  - 2 \alpha \  n^{-1} \cov{ H(\cU_{n,N_n})}{ N_n} \big\}
		\end{split} \\
		&= \sigma^2 + \alpha^2 - 2 \alpha \ \lim_{n\to\infty} n^{-1} \cov{ H(\cU_{n,N_n})}{ N_n}. \label{E:TauSigma3}
\end{align}
The argument relies on the fact that uniform integrability of a sequence $(U_n)_n$ and its convergence in distribution to a limit $U$ imply convergence of the means, viz., $\E{U_n}\to \E{U}$ (using Skorokhod's representation theorem for instance). 

Obviously, the sequence $(n^{-1} Z_n^2 )_n$ is  u.i. Moreover, $(n^{-1} X_n^2)_n$ (resp.\ $(n^{-1} Y_n^2)_n$) converges in distribution to $\tau^2 Z_0$ (resp.\ $\sigma^2 Z_0$), where $Z_0$ has a standard normal distribution. Hence, it is enough to verify the uniform integrability of $( n^{-1} X_n^2 )_n$ and $( n^{-1} Y_n^2 )_n$ in order to conclude \eqref{E:TauSigma1} to \eqref{E:TauSigma3}. 
 
The sequence $(n^{-1} Y_n^2)_n$ is u.i. by the moment conditions given in Lemma~\ref{L:Moments} (take $p=4$). In order to verify that $ ( n^{-1} X_n^2 )_n$ is u.i., it is enough to use the representation given in \eqref{E:LILBinomial1} to \eqref{E:LILBinomial3} and show that the square of the remainder, which is given in \eqref{E:LILBinomial2} and \eqref{E:LILBinomial3}, when divided by $n$, is u.i.

Here, the arguments used in the proof of Proposition \ref{P:LILBinomialError} (see in particular \eqref{E:LILBinomialErrorDiff1}) and in Proposition~\ref{P:LILBinomialError2} (see in particular \eqref{E:LILBinomialErrorExp1a} and \eqref{E:LILBinomialErrorExp1b})  lead to the u.i. of $n^{-1}$ times the square of the term in \eqref{E:LILBinomial2}, we skip the details. Clearly, Proposition~\ref{P:LILBinomialError2} applies to $n^{-1}$ times the square of the term in \eqref{E:LILBinomial3}.
 This shows \eqref{E:TauSigma1} to \eqref{E:TauSigma3} and in particular $\lim_{n\to\infty}  n^{-1} \cov{ H(\cU_{n,N_n})}{ N_n} = \alpha$.

Moreover, the covariance can be expressed in terms of the above introduced martingale difference sequence (note that $\EE{ \Delta'(z,n) | \cF_z } = \cP_n( Q_z ) - \EE{ \cP_n( Q_z ) }$). We have
\begin{align*}
	&n^{-1}	\cov{ H(\cU_{n,N_n})}{ N_n} \\
	&= 	n^{-1}	 \EE{ \big( \sum_{z\in B_n} \E{ \Delta(z,n) | \cF_z } \big) \big( \sum_{y\in B_n} \cP_n(Q_y) - \E{  \cP_n( Q_y ) } \big) } \\
		&= 	n^{-1} \sum_{z\in B_n} 	 \E{  \Delta(z,n) \  \cP_n(Q_z) } \\
		&= 	n^{-1} \sum_{z\in B_n} 	\E{ \Delta(z,\infty) \ \cP(Q_z) } \\
		&\quad +n^{-1}  \sum_{z\in B_n}  \E{  \Delta(z,n)  \ \cP_n(Q_z) -  \Delta(z,\infty) \ \cP(Q_z) } \\
		&= 	\E{ \Delta(0,\infty) \ \cP(Q_0) } (1+o(1)) + o(1).
		\end{align*}
where the last equality follows because $\# B_n / n \to 1$ and
\begin{align*}	
			 n^{-1} \sum_{z\in B_n}  \E{  \Delta(z,n)  \cP_n(Q_z) -  \Delta(z,\infty) \cP(Q_z) } \rightarrow 0, \quad (n\to\infty),
\end{align*}		
 by an application of the H{\" o}lder and Cauchy-Schwarz inequality.
Hence, $\E{\fD_0(\cP,\infty)} \overset{\text{def}}{=} \alpha = \E{ \Delta(0,\infty) \cP(Q_0) } $.

Next, we use $\EE{ \Delta'(0,\infty) | \cF_0} = \cP(Q_0) - \E{ \cP(Q_0) }$ and $\EE{ \EE{ \Delta'(0,\infty) | \cF_0}^2 } = 1$. Consequently, we find with the definition of $\tau$
\begin{align*}
		 \E{ \E{ \Delta(0,\infty) - \alpha \Delta'(0,\infty) | \cF_0 }^2 } = \sigma^2 - 2 \alpha \E{ \Delta(0,\infty) \cP(Q_0) } + \alpha^2 = \sigma^2 - \alpha^2 = \tau^2 .
\end{align*}
This shows the desired relation.
\end{proof}

\subsection{The special case of a one-dimensional domain}
In this section, we consider the special case for the 1-dimensional stretched domain $\cD\times [0,n]$. Given the technical details for the full domain $(W_n)_n$, which stretches in each dimension, it is a routine to verify the LIL for the sequences $(H(\cQ_n)-\E{H(\cQ_n)})_n$. So, we skip the details of the proof of Theorem~\ref{T:1DLIL} here and focus on the SIP. In the proof, we continue to use the established notation for reasons of clarity, so we write $\cP_n, \cU_{n,N_n}$ and $\cU_{n,n}$ instead of $\cQ_n$.

\begin{proof}[Proof of Theorem~\ref{T:1DSIP}]
We prove the invariance principle with a standard technique which relies on the Skorokhod embedding for martingales, (see, e.g., \cite{hall1980martingale}). We begin with the Poisson process and introduce the following definitions.
\begin{align*}
		S_{e,n} &= \sum_{z = 0 }^{ n }  \E{ \Delta(z, n)  | \wt\cF_z } = H(\cP_n) - \E{ H(\cP_n) }, \qquad S_n = \sum_{z=1}^{n} F_z.
\end{align*}
Then $S_{e,n} =  S_n + o_{a.s.}((\log n)^{(3/2+\delta} n^{1/p})$ by Proposition~\ref{P:OrderRemainderPoisson} for each $\delta>0$, which is also valid in the case of the stretched domain with $d=1$. 

Using the Skorokhod martingale embedding, there is a standard Brownian motion $B$ and non-negative random variables $\tau_i$, $i\in\N_+$, such that $( S_n )_n$ equals $ (B(\sum_{i=1}^n \tau_i ))_n $ $a.s.$ and the $\tau_i$ satisfy
\[
		\EE{ \tau_i | \wt \cG_{i-1} } = \E{ F_i^2 | F_1,\ldots,F_{i-1} } \quad a.s.
\]
as well as $\EE{ \tau_i^r | \wt \cG_{i-1} } \le C_r \E{ F_i^{2r} | F_1,\ldots,F_{i-1} }$ $a.s.$ for all $r>1$ for a certain $C_r\in\R_+$ and where the $\sigma$-field $\wt \cG_i$ is generated by $(B(\sum_{k=1}^j \tau_k) : j\le i)$ and by the $B(t)$, $t\in [0,\sum_{k=1}^i \tau_k ]$. Also, $\sum_{k=1}^i \tau_k$ is $\wt \cG_i$-measurable. We show
\begin{align}\label{E:SIP1}
		B\Big(\sum_{i=1}^n \tau_i \Big) = B(n\sigma^2) + \cO_{a.s.}\big( n^{1/4} (\log n)^{1/2} (\log \log n)^{1/4} \big).
\end{align}
This shows then $ S_{e,n}$ equals $B( n \sigma^2 ) + \cO_{a.s.}( n^{1/4} (\log n)^{1/2} (\log \log n)^{1/4} )$. The desired approximation follows then from Lemma~\ref{L:BMContinuity}.
 In order to show \eqref{E:SIP1}, we rely on the classical decomposition
\begin{align}\begin{split}\label{E:SIP2}
		\sum_{i=1}^n \tau_i  - n \sigma^2 &= \sum_{i=1}^n ( \tau_i  - \mathbb{E}[ \tau_i | \wt\cG_{i-1} ] ) +  \sum_{z=1}^n ( \mathbb{E}[ F_z^2 | F_1,\ldots,F_{z-1} ] -  F_z^2 ) \\
		&\quad + \sum_{z=1}^n (F_z^2 - \sigma^2 ).
\end{split}\end{align}
The third term on the RHS of \eqref{E:SIP2} satisfies the LIL from Theorem~\ref{T:GeneralLIL}; note that we need in this step, the uniform bounded moments condition to be satisfied with $p$ at least equal to 8. The first and the second term on the RHS of \eqref{E:SIP2} satisfy the martingale LIL (\cite{stout1970hartman}). This yields $| \sum_{i=1}^n \tau_i - n\sigma^2 | \le C  n^{1/2} (\log \log n)^{1/2}  =: t_n $ $a.s.$, for a certain $C\in\R_+$. Since $t_n = o(n)$, we are in position to apply Lemma~\ref{L:BMContinuity}. More precisely,
\begin{align*}
		\Big| B\big( \sum_{i=1}^n \tau_i \big) - B(n\sigma^2 ) \Big | &\le \max_{k\le n} \sup_{x: |x-k\sigma^2 | \le t_n } |B(x) - B(k\sigma^2 )| \\
		&\le C\sqrt{ t_n \log( n t_n^{-1} ) } = \cO_{a.s.}( n^{1/4} (\log n)^{1/2} (\log\log n)^{1/4}).
\end{align*}
This shows \eqref{E:SIP1}. 

Regarding the binomial process, we can use the decomposition from \eqref{E:LILBinomial1} to \eqref{E:LILBinomial3} and the results from Proposition~\ref{P:LILBinomialError} as well as Proposition~\ref{P:LILBinomialError2}. So that
\begin{align*}
		H( \cU_{n,n} ) - \E{ H(\cU_{n,n}) } &= H( \cU_{n,N_n} ) - \E{ H( \cU_{n,N_n} ) } - \alpha(N_n - n) \\
		&\quad + \cO_{a.s.}( n^{1/4 + 1/p} (\log n)^2 ) + \E{H(\cU_{n,N_n})  - H(\cU_{n,n}) },
\end{align*}
where $\E{H(\cU_{n,N_n})  - H(\cU_{n,n}) } = \cO(n^{1/4})$ and where the term $H( \cU_{n,N_n} ) - \E{ H( \cU_{n,N_n} ) } - \alpha(N_n - n) $ equals $\sum_{z=0}^n F'_z = \sum_{z=0}^{n} \E{ \Delta(z,\infty) - \alpha \ \Delta'(z,\infty)  | \cF_{z} }$ plus an error, which is of order $o_{a.s.}( (\log n)^2 n^{1/p} )$ by Proposition~\ref{P:OrderRemainderPoisson} and Lemma~\ref{L:PoissonPointsRem}. Repeating the above reasoning with $\sum_{i=1}^n F'_z$ instead completes the proof in the binomial case.
\end{proof}

\subsection{Verification of the examples}\label{Section_VerificationExample}

\begin{proof}[Proof of Example~\ref{Ex:EC}]
Let $r$ be an upper bound on the diameter of a simplex in $\cK$. Then $S^*$ is at most $r$. Consider a point $z\in\R^d$ and a point cloud $P$. Let $m = \#( P\cap B(z,S^*))$. Then the Euler characteristic computed from these $m$ points is at most $\sum_{k=0}^{m-1} \binom{m}{k+1}$. Lemma~\ref{L:MomentsEuler} shows that this is at most $C^* m^{1/2} 2^m$ for a certain $C^*\in\R_+$. In particular, the Euler characteristic satisfies model~\ref{itm:model1}.
\end{proof}

\begin{proof}[Proof of Example~\ref{Ex:kNN}]
For simplicity, we only consider the case for 2 dimensions, see also \cite{penrose2001central} for generalizations to higher dimensions or for the 1-dimensional case. We show that $\p( \wt S > r )\le c_1 \exp(-c_2 r )$ for two constants $c_1,c_2\in\R_+$ whenever $r\in (0, n^{1/d}/2 )$. It suffices to consider the case where $x=0$. The (number of) points inside $Q_0$, $V_1,\ldots,V_k$ is also not the main difficulty, as we will see. Depending on the parameters $y,n$, we need to distinguish two cases. 

\textit{Case 1. $B(0,r)\subset \cC_{y,n}$}. We follow the ideas of \cite{penrose2001central} and construct a neighborhood of $0$ from 6 disjoint equilateral triangles $T_j$, $j\in\{1,\ldots,6\}$, with edge length $r/4$ as in the left panel of Figure~\ref{fig:kNN}. Then each $T_j$ has Lebesgue measure $\sqrt{3} r^2 / 16 = \lambda(r)$. We have
\begin{align*}
			&\p( T_j \text{ contains at most $k$ points of } \cP \text{ for one } j\in\{1,\ldots,6\} ) \\
			&\le 6 \p( T_1 \text{ contains at most $k$ points of } \cP ) \\
			&\le 6 (k+1) (1+\lambda(r))\exp(-\lambda(r)) \le c_1 \exp( -c_2 r),
\end{align*}
for certain constants $c_1,c_2\in\R_+$. One finds that given each $T_j$ contains at least $k+1$ points, the radius of stabilization is at most $4 (r/4) = r$. We omit the details here because they will be explained in detail in the next case.

\textit{Case 2. $B(0,r)\not\subset \cC_{y,n}$}. In this case, $B(0,r)$ either intersects with one or with two edges because $r< 2^{-1} n^{1/d}$, see the middle and the right panel of Figure~\ref{fig:kNN}. We use again the 6 disjoint equilateral triangles with edge length $r/4$. This time, we divide each equilateral triangle $T_j$ in two disjoint isosceles triangles. Then the Lebesgue measure of each isosceles triangle is $\lambda(r)/2$. 

We begin with the situation in the middle panel and consider the 6 isosceles triangles with edge length $r/4$ in the lower half (given in brown in the figure), $T'_1, \ldots,T'_6$, say. Note that in general the blue and green triangles in the upper half can intersect with $\R^d \setminus \cC_{y,n}$ depending on the distance 0 to the boundary of  $\cC_{y,n}$.

Define the event
\[
	A= \{T'_i \setminus Q_0 \text{ contains at least $k+1$ points for $i\in\{1,\ldots,6\}$} \ \}
\]
Then, $\p( A^c ) \le c_1 \exp(-c_2 r)$ for certain constants $c_1,c_2\in\R_+$. We show that given the event $A$, the radius of stabilization is at most $9r$. 

Plainly, given $A$, 0 has $(k+1)$ points within distance $r/4$ from $\cP|_{\cC_{y,n}}$ because all 6 isosceles triangles are entirely contained within $\cC_{y,n}$ and each of these triangles contains at least $(k+1)$ points.

 Conversely, given $A$, let $0$ be a kNN of a point $z\in ([\cP|_{\cC_{n,y}}\setminus Q_0] \cup \{V_1,\ldots,V_k\}) $. 
If the shortest path of $z$ to 0 intersects with one of the two equilateral triangles at the top (in green), then $z$ lies in the induced convex cone within a distance of $2r$ (otherwise $z$ lies outside of $
\cC_{n,y}$, one can show this with elementary calculations).
 
Otherwise (if the path does not intersect with the green triangles), the path must pass through one of the six lower isosceles triangles (in brown) or through one of the two neighboring isosceles triangles (in blue). By construction each of the six lower isosceles triangles contains $(k+1)$ points. Write $C$ for the two blue equilateral triangles.

Then it holds: If $z$ is not contained in $\cup_{i} T'_{i} \cup C $ and if the shortest path of $z$ does not pass through the two top green triangles, then $0$ is no kNN of $z$. Indeed, if the path of $z$ passes through $C$ but $z\notin C$, then $z$ lies closer to $(k+1)$ points in a brown adjacent triangle than to 0. If the path of $z$ passes through one of the $T'_i$ but $z\notin T'_i$, then the $(k+1)$ points from this $T'_i$ lie closer to $z$ than 0.

Hence, if $0$ is a kNN of $z$, then either $z\in \cup_{i} T'_{i} \cup C$ or the shortest path of $z$ to 0  passes through the two top triangles and $z$ lies in $B(0,2r)$. So in both cases $z \in B(0,2r)$ and has $(k+1)$ points in distance $2r+r/4$.

Next, let $\{z,y\}$ be an edge which is removed from the graph due to the additional point $0$, w.l.o.g., let $0$ be a kNN of $z$, so that $y$ is a $(k+1)$-nearest neighbor of $z$ when adding 0. Then $y \in B(0,2(2r+r/4) )$. If additionally $z$ is not a kNN of $y$, then the edge $\{z,y\}$ is removed. Otherwise, $z$ is a kNN of $y$ and $y$ has $k$ points within distance $2(2r+r/4)$. Consequently, the configuration of the points in $\cC_{n,y}$ but outside of $B(0,4(2r+r/4))$ is irrelevant for the choice of whether an edge $\{z,y\}$ is removed from the graph. This completes the case for the situation in the middle panel of Figure~\ref{fig:kNN} and the radius of stabilization is at most $9r$. 

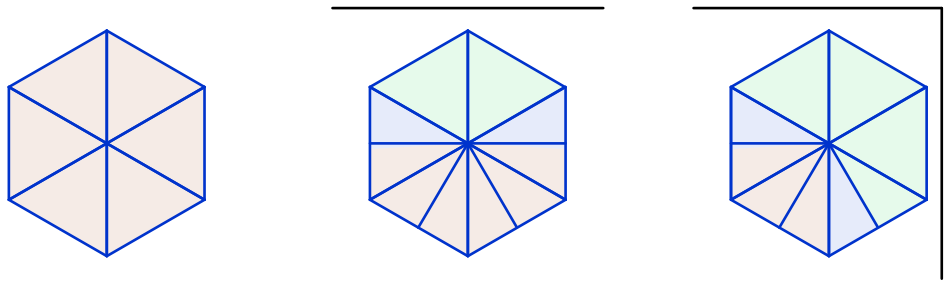
\begin{figure}
\centering
\begin{tikzpicture}[line cap=round,line join=round,>=triangle 45,x=.3cm,y=.3cm]
\clip(-2.8,-.66) rectangle (50,16.25);
\fill[line width=1pt,color=redCol,fill=blueCol,fill opacity=0.1] (1.66,10.5) -- (6,8) -- (6,13) -- cycle;
\fill[line width=1pt,color=redCol,fill=blueCol,fill opacity=0.1] (6,13) -- (6,8) -- (10.33,10.5) -- cycle;
\fill[line width=1pt,color=redCol,fill=blueCol,fill opacity=0.1] (10.33,10.5) -- (6,8) -- (10.33,5.5) -- cycle;
\fill[line width=1pt,color=redCol,fill=blueCol,fill opacity=0.1] (10.33,5.5) -- (6,8) -- (6,3) -- cycle;
\fill[line width=1pt,color=redCol,fill=blueCol,fill opacity=0.1] (6,3) -- (6,8) -- (1.66,5.5) -- cycle;
\fill[line width=1pt,color=redCol,fill=blueCol,fill opacity=0.1] (1.66,10.5) -- (6,8) -- (1.66,5.5) -- cycle;

\fill[line width=1pt,color=redCol,fill=blueCol,fill opacity=0.1] (22,8) -- (26.33,5.5) -- (26.33,7.75) -- cycle;
\fill[line width=1pt,color=redCol,fill=greenCol,fill opacity=0.1] (17.66,10.5) -- (22,8) -- (22,13) -- cycle;
\fill[line width=1pt,color=redCol,fill=blueCol,fill opacity=0.1] (17.66,5.5) -- (22,3) -- (22,8.0) -- cycle;
\fill[line width=1pt,color=redCol,fill=blueCol,fill opacity=0.1] (22,8) -- (22,3.0) -- (26.33,5.5) -- cycle;
\fill[line width=1pt,color=redCol,fill=greenCol,fill opacity=0.1] (22,13.0) -- (26.33,10.5) -- (22,8) -- cycle;
\fill[line width=1pt,color=redCol,fill=blueCol,fill opacity=0.1] (17.66,7.75) -- (22,8.0) -- (17.66,5.5) -- cycle;

\fill[line width=1pt,color=redCol,fill=redCol,fill opacity=0.1] (17.66,10.5) -- (22,8.0) -- (17.66,7.75) -- cycle;
\fill[line width=1pt,color=redCol,fill=redCol,fill opacity=0.1] (22,8) -- (26.33,7.75) -- (26.33,10.5) -- cycle;

\fill[line width=1pt,color=redCol,fill=greenCol,fill opacity=0.1] (38,8) -- (42.33,5.5) -- (42.33,10.5) -- cycle;
\fill[line width=1pt,color=redCol,fill=greenCol,fill opacity=0.1] (33.66,10.5) -- (38,8) -- (38,13) -- cycle;
\fill[line width=1pt,color=redCol,fill=blueCol,fill opacity=0.1] (33.66,5.5) -- (38,3) -- (38,8.0) -- cycle;

\fill[line width=1pt,color=redCol,fill=greenCol,fill opacity=0.1] (38,8) -- (40.15,4.31) -- (42.33,5.5) -- cycle;
\fill[line width=1pt,color=redCol,fill=redCol,fill opacity=0.1] (38,8) -- (38,3.0) -- (40.15,4.31) -- cycle;

\fill[line width=1pt,color=redCol,fill=greenCol,fill opacity=0.1] (38,13.0) -- (42.33,10.5) -- (38,8) -- cycle;
\fill[line width=1pt,color=redCol,fill=redCol,fill opacity=0.1] (33.66,10.5) -- (38,8.0) -- (33.66,7.75) -- cycle;
\fill[line width=1pt,color=redCol,fill=blueCol,fill opacity=0.1] (33.66,7.75) -- (38,8.0) -- (33.66,5.5) -- cycle;

\draw [line width=1pt,color=redCol] (1.66,10.5)-- (6,8);
\draw [line width=1pt,color=redCol] (6,8)-- (6,13);
\draw [line width=1pt,color=redCol] (6,13)-- (1.66,10.5);
\draw [line width=1pt,color=redCol] (6,13)-- (6,8);
\draw [line width=1pt,color=redCol] (6,8)-- (10.33,10.5);
\draw [line width=1pt,color=redCol] (10.33,10.5)-- (6,13);
\draw [line width=1pt,color=redCol] (10.33,10.5)-- (6,8);
\draw [line width=1pt,color=redCol] (6,8)-- (10.33,5.5);
\draw [line width=1pt,color=redCol] (10.33,5.5)-- (10.33,10.5);
\draw [line width=1pt,color=redCol] (10.33,5.5)-- (6,8);
\draw [line width=1pt,color=redCol] (6,8)-- (6,3);
\draw [line width=1pt,color=redCol] (6,3)-- (10.33,5.5);
\draw [line width=1pt,color=redCol] (6,3)-- (6,8);
\draw [line width=1pt,color=redCol] (6,8)-- (1.66,5.5);
\draw [line width=1pt,color=redCol] (1.66,5.5)-- (6,3);
\draw [line width=1pt,color=redCol] (1.66,10.5)-- (6,8);
\draw [line width=1pt,color=redCol] (6,8)-- (1.66,5.5);
\draw [line width=1pt,color=redCol] (1.66,5.5)-- (1.66,10.5);
\draw [line width=1pt,color=redCol] (22,8.0)-- (26.33,5.5);
\draw [line width=1pt,color=redCol] (26.33,5.5)-- (26.33,10.5);
\draw [line width=1pt,color=redCol] (26.33,10.5)-- (22,8);
\draw [line width=1pt,color=redCol] (17.66,10.5)-- (22,8.0);
\draw [line width=1pt,color=redCol] (22,8)-- (22,13);
\draw [line width=1pt,color=redCol] (22,13)-- (17.66,10.5);
\draw [line width=1pt,color=redCol] (17.66,5.5)-- (22,3.0);
\draw [line width=1pt,color=redCol] (22,3.0)-- (22,8.0);
\draw [line width=1pt,color=redCol] (22,8.0)-- (17.66,5.5);
\draw [line width=1pt,color=redCol] (22,8.0)-- (22,3.0);
\draw [line width=1pt,color=redCol] (22,3)-- (26.33,5.5);
\draw [line width=1pt,color=redCol] (26.33,5.5)-- (22,8);
\draw [line width=1pt,color=redCol] (22,13)-- (26.3,10.5);
\draw [line width=1pt,color=redCol] (26.3,10.5)-- (22,8);
\draw [line width=1pt,color=redCol] (22,8)-- (22,13);
\draw [line width=1pt,color=redCol] (17.66,10.5)-- (22,8);
\draw [line width=1pt,color=redCol] (22,8)-- (17.66,5.5);
\draw [line width=1pt,color=redCol] (17.66,5.50)-- (17.66,10.5);
\draw [line width=1pt,color=redCol] (22,8)-- (26.33,8);
\draw [line width=1pt,color=redCol] (17.66,8)-- (22,8);
\draw [line width=1pt,color=redCol] (22,8)-- (24.15,4.31);
\draw [line width=1pt,color=redCol] (22,8)-- (19.80,4.26);
\draw [line width=1pt] (16,14)-- (28,14);
\draw [line width=1pt,color=redCol] (33.66,5.5)-- (33.66,10.5);
\draw [line width=1pt,color=redCol] (38,8.0)-- (42.33,5.5);
\draw [line width=1pt,color=redCol] (42.33,5.5)-- (42.33,10.5);
\draw [line width=1pt,color=redCol] (42.33,10.5)-- (38,8);
\draw [line width=1pt,color=redCol] (33.66,10.5)-- (38,8.0);
\draw [line width=1pt,color=redCol] (38,8)-- (38,13);
\draw [line width=1pt,color=redCol] (38,13)-- (33.66,10.5);
\draw [line width=1pt,color=redCol] (33.66,5.5)-- (38,3.0);
\draw [line width=1pt,color=redCol] (38,3.0)-- (38,8.0);
\draw [line width=1pt,color=redCol] (38,8.0)-- (33.66,5.5);
\draw [line width=1pt,color=redCol] (38,8.0)-- (38,3.0);
\draw [line width=1pt,color=redCol] (38,3)-- (42.33,5.5);
\draw [line width=1pt,color=redCol] (42.33,5.5)-- (38,8);
\draw [line width=1pt,color=redCol] (38,13)-- (42.3,10.5);
\draw [line width=1pt,color=redCol] (42.3,10.5)-- (38,8);
\draw [line width=1pt,color=redCol] (38,8)-- (38,13);
\draw [line width=1pt,color=redCol] (33.66,10.5)-- (38,8);
\draw [line width=1pt,color=redCol] (38,8)-- (33.66,5.5);
\draw [line width=1pt,color=redCol] (33.66,5.50)-- (33.66,10.5);
\draw [line width=1pt,color=redCol] (33.66,8)-- (38,8);
\draw [line width=1pt,color=redCol] (38,8)-- (40.15,4.31);
\draw [line width=1pt,color=redCol] (38,8)-- (35.80,4.26);
\draw [line width=1pt] (32,14)-- (43,14);
\draw [line width=1pt] (43,14)-- (43,2);
\end{tikzpicture}
\caption{In the left panel $B(0,r)$ is contained in $\cC_{y,n}$, in the middle and right panel $B(0,r)$ is not contained in $\cC_{y,n}$. In the left panel, we construct a neighborhood with equilateral triangles of edge length $r/4$. In the middle and right panel, we use additionally a refinement by dividing some equilateral triangles in isosceles triangles.}
\label{fig:kNN}
\end{figure}
The arguments are very similar for the right panel. Consider the three brown isosceles triangles with edge length $r/4$, $T''_i$, say. 

Once more, set $A = \{ T''_i\setminus Q_0 \text{ contains at least $k+1$ points for } i\in \{1,\ldots,3\} \}$. Then $\p(A)$ is at least $1-c_1 \exp(-c_2 r)$, so that we can restrict our considerations conditional on the event $A$. Clearly, 0 contains $(k+1)$ points within distance $r/4$. Moreover, arguing as before if 0 is a kNN of a $z$, then $z$ lies within a distance of $m r$, for some multiple $m\in\R_+$ and there are $(k+1)$ points within distance of $mr+r/4$.

Furthermore, let $\{z,y\}$ be an edge which is removed, and let $0$ be a kNN of $z$. Then $y\in B(0,2(mr+r/4))$. If additionally $z$ is not a kNN of $y$, then $\{z,y\}$ is removed; otherwise $y$ has $k$ points within distance $2(mr+r/4)$. So, the configuration of the points in $\cC_{y,n}$ but not in $B(0,2(mr+r/4))$ is irrelevant for the choice of whether an edge $\{z,y\}$ is removed from the graph. Consequently, in the situation in the right panel of Figure~\ref{fig:kNN}, the radius of stabilization is at most $(m+1)r$. 
\end{proof}

\addcontentsline{toc}{section}{Acknowledgments}
\textbf{Acknowledgements.} I thank the anonymous referee for her/his helpful comments that improved the manuscript.

\newpage
\appendix
\section{Appendix}\label{Appendix}

\subsection{General results}\label{A:General}

\begin{lemma}\label{L:MomentsEuler}
Let $m \in \N$. There is a constant $C \in\R_+$, which does not depend on $m$, such that
\[
		\binom{m}{k} = \frac{m!}{k!(m-k)!} \le C  \frac{2^m}{m^{1/2} },  \quad \forall k\in\{0,\ldots,m \}.
\]
Moreover, if $\lambda\in\R_+$ and $X \sim \poi(\lambda)$, then $\EE{ |\sum_{k=0}^X \binom{X}{k} |^q + e^{q X} }< \infty$ for all $q\in\R_+$. 
\end{lemma}
\begin{proof}
The result relies on Stirling's formula
\[
		\sqrt{2\pi} n^{n+1/2} e^{-n} \le n! < e n^{n+1/2} e^{-n} \text{ for } n\in\N_+.
\]		
It is well known that the binomial coefficient is maximal at $m/2$ if $m$ is even and at $(m+1)/2$ if $m$ is odd. Thus, if $m$ is even,
\begin{align*}
		 \frac{m!}{k!(m-k)!} &\le  \frac{m!}{((m/2)!)^2} \le \frac{m! \ 2^{m+1} }{2\pi \ m^{m+1} \ e^{-m} } \le \frac{e \ 2^{m} }{\pi \ m^{1/2} }.
\end{align*}
A similar result is valid if $m$ is odd,
\begin{align*}
		 \frac{m!}{k!(m-k)!} &\le  \frac{m!}{ (\frac{m+1}{2})! (\frac{m-1}{2})!  } \le \frac{e^{3/2} \ 2^{m} }{\pi \ m^{1/2} + o(1)  }.
\end{align*}
 The claim regarding the moments follows immediately because $\EE{ e^{\delta X } } = \exp( \lambda(e^\delta -1 ) )$ is finite for all $\delta< \infty$. This completes the proof.
\end{proof}

\begin{lemma}[Uniform bounded moments]\label{L:Moments} Let $\rho \in (0,1/2)$.
\begin{itemize}
\item [(a)] There is a constant $C$, which only depends on the dimension $d$, such that for all $y\in\R^d$ and $r\in\R_+$
\begin{align}\begin{split}\label{E:Moments1}
		&\sup_{m\in [n-n^{1/2 + \rho }, n+n^{1/2 + \rho}] } \p( \cU_{n,m} \cap B(y,r)  \neq \cP \cap B(y,r) \ | \ \cU_{n,m} \cap B(y,r)  ) \\
		&\le C \Big(\frac{r^d}{ n}\Big)^{1/2 - \rho}
\end{split}\end{align}
with probability one.

\item [(b)] Assume model~\ref{itm:model1}. The uniform bounded moments condition \eqref{C:UniformBoundedMoments} is also satisfied for model~\ref{itm:model1} for any positive exponent. Moreover, let $p\in\N$, then
\begin{align}
\begin{split}\label{E:Moments2a}
		&\sup_{n\in\N} \sup_{y\in B_n} \sup_{m\in [n-n^{1/2 + \rho},n+n^{1/2 + \rho}]} \mathbb{E}\Big[ \Big| H( (\cU_{n,m} \cap B(y,S^*)) \cup \{y\} ) \\
		&\qquad\qquad\qquad\qquad\qquad\qquad\qquad - H(\cU_{n,m} \cap B(y,S^*) ) \Big|^p \Big] < \infty, \end{split} \\ 
			&	\E{ \Big| H( (\cP \cap B(0,S^*)) \cup \{0\} ) - H( \cP \cap B(0,S^*) ) \Big|^p }< \infty, \label{E:Moments2b} \\
				&\sup_{n\in\N} \max_{z\in B''_n} \E{ |\Delta(z,n)|^p } + \E{ | \Delta(z,\infty)| ^{ p} } < \infty. \label{E:Moments2c}
\end{align}

\item [(c)] Assume model~\ref{itm:model2}. Let $p\in\R_+$ be as in Condition~\eqref{C:UniformBoundedMoments} for the Poisson sampling scheme and let $S_0$ be given by \eqref{D:S0}. Then
\begin{align}
	&\sup_{n\in\N, z\in B_n}	\E{ | H(\cP_n) - H(\cP''_{n,z} ) |^{p} } < \infty, \label{E:Moments3a} \\
	&\sup_{n\in\N, z\in B_n, r\in\R_+} \E{ |H(\cP_n \cap B(z,r) )  - H( \cP''_{n,z} \cap B(z,r) ) |^{p} } < \infty, \label{E:Moments3b} \\
		&\sup_{n\in\N, z\in B_n, r\in\R_+} \E{ |  H(\cP\cap B(z,r)) - H(\cP''_z\cap B(z,r)) |^{p} } < \infty, \label{E:Moments3c}\\
				&\sup_{n\in\N, z\in B_n} \E{ | H( (\cP - z) \cap B(0,S_0)) - H( (\cP''_z - z ) \cap B(0,S_0))|^{p} } < \infty. \label{E:Moments3d}
\end{align}
\end{itemize}
\end{lemma}
\begin{proof}
We begin with (a). Let $N_n$ be the number of points of $\cP$ in $B_n$. Then on $B(y,r)$ the point processes can differ in at most $|N_n - m|$ i.i.d.\ points, each of these points falls in $B(y,r)$ with a probability of order $r^d n^{-1}$. There is a constant $C$, which only depends on $d$, such that the left-hand side of \eqref{E:Moments1} is less than
\begin{align*}
	\frac{C r^d}{n} \E{ |N_n - m | } \le \frac{C r^d}{n} \Big( \E{ |N_n - n| } + |n-m| \Big) \le C r^d (n^{-1/2} + n^{-1/2 + \rho} ).
\end{align*}
This shows \eqref{E:Moments1}.

We come to (b). Consider the uniform bounded moments condition \eqref{C:UniformBoundedMoments}. Using the hard-thresholded stabilization, it is enough to compute $\EE{ | H( [\cU_{A,m} \cap B(0,S^*)]\cup \{0\} )|^{p'} }$ and $\EE{ |H( \cU_{A,m} \cap B(0,S^*))|^{p'} }$ for an $A\in\cB$ and $m \in [1/2 |A|, 3/2 |A| ]$ as well as $p'\in\R_+$. This amounts to compute $\EE{ 2^{\delta N} }$, for some $\delta>0$ and $N$ follows a binomial distribution of length $m$ and a success probability of order $|A|^{-1}$. Hence,
\[
			\E{ 2^{\delta N} } = \sum_{k=0}^m \frac{m!}{k!(m-k)!} \Big(\frac{c}{|A|}\Big)^k \Big(1- \frac{c}{|A|}\Big)^{m-k} 2^{\delta k} = e^{ 3c (2^\delta - 1) / 2},
\]
where we use that $|A|^{-1} \le 3/2 \ m^{-1}$ and the right-hand side only depends on $\delta$. 

We continue with the inequalities in \eqref{E:Moments2a} to \eqref{E:Moments2c}. Let $m,n$ be arbitrary but fixed, $m\in [n-n^{1/2+\rho},n+n^{1/2+\rho}]$. Let $y\in B_n$. Let $M$ be the number of points of $\cU_{n,m}$ in $Q_y^{(\delta)}$. By assumption,
\[
		|H( (\cU_{n,m} \cap B(y,S^*) )| \vee |  H( (\cU_{n,m} \cap B(y,S^*) ) \cup \{y\} ) | \le  ce^{qM},		
\]
for some $q\in\N$.

 By construction, $M$ is stochastically dominated by 1 plus a binomial random variable of length $n+n^{1/2+\rho}$ and a success probability proportional to $n^{-1}$. Reasoning as before leads to \eqref{E:Moments2a}. \eqref{E:Moments2b} and \eqref{E:Moments2c} follow with the same reasoning. 

Finally, regarding (c), we can use a de-Poissonization argument, it is a routine to verify the first three statements \eqref{E:Moments3a} to \eqref{E:Moments3c} with Condition~\eqref{C:UniformBoundedMoments}; the proof works similar as in \cite{penrose2001central} Lemma 3.1 and Lemma 4.1. \eqref{E:Moments3d} follows from the polynomial growth of $H( (\cP-z) \cap B(0,n) )$ along with the exponential decay of $\p( S_0 \in (n-1,n] )$ which is at most $\p( S_0 > n-1)$.
\end{proof}

\begin{lemma}\label{L:LILTool}
Let $(X_n)_n$ be a sequence of random variables satisfying 
$$
(\log n )^{-\alpha} n^{-\beta} \| \max_{k\le n} |X_k| \|_q \le C (\log n)^{-(1+\delta)}
$$
for constants $q\ge 1$ and $C,\alpha,\beta, \delta \in\R_+$ (which are independent of $n$).

Then $\limsup_{n\to\infty} (\log n )^{-\alpha} n^{-\beta} |X_n| \le \wt C$ $a.s.$ for some $\wt C\in\R_+$ (independent of $n$).
\end{lemma}
\begin{proof}
Clearly, $\sum_{j\in\N} (\log (2^j) )^{-\alpha} 2^{- j \beta} \| \max_{k\le 2^j} |X_k| \|_q < \infty$. Hence, $ \max_{k\le 2^j} |X_k|  = \cO_{a.c.}[ (\log (2^j) )^{\alpha}  (2^j)^{\beta} ]$. In particular, there is a constant $C^*\in\R_+$ which satisfies
$$
		\limsup_{j\to\infty} \	(\log (2^j) )^{-\alpha}  (2^j)^{-\beta} \max_{k\le 2^j} |X_k| \le C^* \quad a.s.
$$
Let $n\ge 4$ and write $\ol j$ for the largest integer such that $2^{\ol j -1 } < n \le 2^{\ol j}$. Then $\ol j \ge 2$ and
\begin{align*}
	(\log n )^{-\alpha} n^{-\beta} \max_{k\le n} |X_k| &\le \frac{ (\log (2^{\ol j}) )^{\alpha}}{ (\log n )^{\alpha}} \ \frac{ 2^{ \beta \ol j}}{n^\beta} \ \frac{ \max_{k\le 2^{\ol j}} |X_k| }{ (\log (2^{\ol j} ) )^{\alpha} 2^{\beta \ol j} }   \\
	&\le \frac{\ol j ^\alpha \ 2^\beta}{(\ol j -1 )^{\alpha} }  \frac{ \max_{k\le 2^{\ol j}} |X_k| }{ (\log (2^{\ol j} ) )^{\alpha} 2^{\beta \ol j} }   \le 2^{\alpha+\beta} \frac{ \max_{k\le 2^{\ol j}} |X_k| }{ (\log (2^{\ol j} ) )^{\alpha} 2^{\beta \ol j} }  . 
\end{align*}
Consequently,
$$
	\limsup_{n\to\infty} \ (\log n )^{- \alpha} n^{-\beta} |X_n| \le  \limsup_{n\to\infty} \ (\log n )^{-\alpha} n^{-\beta} \max_{k\le n } |X_k| \le 2^{\alpha+\beta} C^* < \infty
$$
with probability 1.
\end{proof}

\subsection{Deferred results for the law of the iterated logarithm}\label{A:LIL}

\begin{proof}[Proof of Lemma~\ref{L:Covariance}]
First, we show the amendment. Define $r=d(I,J)/2$. Put 
\begin{align}\label{E:Covariance1}
	\wt F_z = \E{ H(\cP \cap B(z,r )) - H(\cP''_z \cap B(z,r)) |\cF_z }
\end{align}
 and $X' = G( \sum_{z\in I} h(\wt F_z) )$ and $Y' = \wt G( \sum_{z\in J} h(\wt F_z) )$. Then
\begin{align*}
		\cov{X}{Y} &= \cov{X-X'}{Y-Y'} + \cov{X-X'}{Y'} + \cov{X'}{Y-Y'}
\end{align*}
because $\cov{X'}{Y'}=0$. Moreover, 
\begin{align*}
		| \cov{X-X'}{Y-Y'} | &\le \E{ |X-X'| |Y-Y'| \1{F_z \neq \wt F_z \text{ for one }z\in I\cup J } }  \\
		&\quad + \E{ |X-X'| \1{F_z \neq \wt F_z \text{ for one }z\in I } }  \\
		&\qquad \cdot \E{  |Y-Y'| \1{F_z \neq \wt F_z \text{ for one }z\in  J } } 
		\end{align*}
	and similarly for $\cov{X-X'}{Y'}$ and $\cov{X'}{Y-Y'}$. In both models, there are constants such that $\p( F_z \neq \wt F_z ) \le c_1 \exp( -c_2 r )$. This shows the amendment because both $G$ and $\wt G$ are bounded functions.
	
Consider now the special case where $G=\text{id}$ and where the index sets are singletons (the fact that $G$ is not bounded is irrelevant). If $h$ is the identity, then $\wt \sigma^2$ is simply $\EE{F_0^2}$. In the other case, use the definition in \eqref{E:Covariance1} with $r = d(z,0)/2$ and $\delta>0$
\begin{align*}
		&\sum_{z\in\Z^d} |\cov{ h(F_0)}{ h(F_z) }| \\
		&\le \EE{ h(F_0)^2 } + \sum_{z \in\Z^d:  z \neq 0} \Big| \E{ h(F_0) h( F_z) \1{ F_0 \neq \wt F_0, F_z \neq \wt F_z } } \Big| \\
		&\le \EE{ h(F_0)^2 } + \sum_{z \in\Z^d:  z \neq 0} \E{ h(F_0)^{2+\delta} }^{2/(2+\delta)} c'_1 \exp( -c'_2  \ d(z,0)/2 ) < \infty.
\end{align*}
Thus, the definition of $\wt \sigma^2$ is meaningful in both cases. This completes the proof.
\end{proof}

\begin{proof}[Proof of Proposition~\ref{P:NormalApproximation}]
The proof is based on approximation techniques between the characteristic functions and a suitable partitioning of the index sets into large and small blocks (which is a quite common technique when working with random fields).

Let $\epsilon\in (0,1/2)$ and $n\ge 2^{1/\epsilon}$. Set $p=p(n)\coloneqq \floor{ n^{1/2}}$, $q(n) \coloneqq \floor{n^{1/2 -\epsilon}}$ and $k(n)\coloneqq \floor{ (2n+1)/(2p+1+q) }$. The set $[-n,n]\cap \N$ contains $k$ sets $I_1,\ldots,I_k$ of cardinality $2p+1$ with distance $q$ between them. For each $i\in \{1,\ldots,k\}^d$, define a cube $J_i = I_{i_1}\times \ldots\times I_{i_d}$ with cardinality $\# J_i = (2p+1)^d$. Then one sees
\begin{align}\begin{split}\label{E:NormalApproximation0}
			1 & \ge \frac{ \# (\bigcup_{i} J_i )}{ \# D_n } = \frac{ k^d \ (2p+1)^d }{ (2n+1)^d } \\
			&\ge \Big( 1 -  \frac{4n^{1/2} }{2n+1} \Big)^d \ \Big( 1 -  \frac{n^{1/2 - \epsilon}}{n^{1/2} } \Big)^d\ge 1 - 2d(2n^{-\epsilon}).
\end{split}\end{align}
Moreover, let
\[
	\xi_i \coloneqq \sum_{z\in J_i} h(F_z), \quad i\in\{1,\ldots,k\}^d,
	\]
and let the family $\{ \xi'_i: i\in\{1,\ldots,k\}^d \}$ be i.i.d.\ with distribution $\xi_1$. (Note that the $\xi_i$ also depend on $n$.) We introduce the following random variables 
\begin{align*}
			Z_n &\coloneqq \frac{ \sum_{z\in D_n} h(F_z) }{ \wt \sigma (\# D_n)^{1/2} }, \qquad		Z'_n \coloneqq \frac{ \sum_{i} \xi_i }{ \wt \sigma (\# D_n)^{1/2} },\\
			\wt Z_n &\coloneqq \frac{ \sum_{i} \xi'_i }{ \wt \sigma (\# D_n)^{1/2} },\qquad			Z^*_n \coloneqq \frac{ \sum_{i}\xi'_i }{ k^d Var( \xi_1 )^{1/2} }.
\end{align*}
In particular, the variances of $Z_n, Z'_n, \wt Z_n, Z^*_n$ are uniformly bounded in $n$. 

Let $T>0$. Esseen's theorem (\cite{esseen1945fourier}) states that for two constants $K'_1, K'_2\in\R_+$, which are independent of $T$,
\begin{align*}
		| \p( Z_n \le z) - \Phi(z) | &\le K'_1 \int_{-T}^{T} \frac{1}{t} \ | \EE{ e^{it Z_n} } - \EE{e^{it Z} } | \ \intd{t} 	+ \frac{K'_2}{T},
\end{align*}
where $Z$ is a random variable with a standard normal distribution.
First, consider
\begin{align}\begin{split}\label{E:NormalApproximation1}
		| \E{ \exp( it Z_n ) } - \E{ \exp( it Z'_n) } | &\le Ct \ \E{ |Z_n-Z'_n|^2 }^{1/2} \\
		&= Ct \ \frac{ Var( \sum_{z\in D_n \setminus \cup_{i}  J_i } h(F_z) )^{1/2} }{\wt\sigma (\# D_n)^{1/2} } \le Ct \ n^{-\epsilon/2}
\end{split}\end{align}
because $\sum_y |Cov(h(F_z),h(F_y))| \le C \in\R_+$ for all $z\in\Z^d$ as in the proof of Lemma~\ref{L:Covariance}. Second, one finds
\begin{align}
			\Big| \E{ \exp( it \wt Z_n ) } - \E{ \exp( it Z^*_n) } \Big| &\le Ct \ \E{ \Big( \sum_{i} \xi'_i \Big)^2 \Big( \frac{\wt \sigma (\# D_n )^{1/2} - (k^d \ Var(\xi_1) )^{1/2} }{\wt \sigma (\# D_n)^{1/2} (k^d \ Var(\xi_1) )^{1/2} }		\Big)^2		}^{1/2} \nonumber \\
			&= Ct \Big| 1- \sqrt{ \frac{k^d \ Var(\xi_1) }{\wt\sigma^2 (\# D_n)} }	\Big| \nonumber \\	
					&\le Ct \Big| 1 - \frac{k^d (2p+1)^d }{(2n+1)^d}\Big| + Ct \Big| 1 - \frac{
			Var( \xi_1) }{\wt \sigma^2 (2p+1)^d } \Big| \label{E:NormalApproximation2}\\
			& \le Ct n^{-\epsilon}, \nonumber
\end{align}
where the last inequality follows from \eqref{E:NormalApproximation0} and the observation: Let $J = [-m,m]^d$, then
\begin{align*}
		| \wt \sigma^2 \# J - Var\big(\sum_{z\in J} h(F_z) \big) | &=  \Big| \sum_{z \in J} \sum_{y\notin J} \E{  h(F_y) h(F_z)  } \Big| \\
		&\le \sum_{r=0}^m  \sum_{z \in J, y\notin J, \| y-z\| = r} \Big| \E{  h(F_y) h(F_z)  } \Big| \\
		&\qquad + \sum_{r=m+1}^\infty  \sum_{z \in J, y\notin J, \| y-z\| = r} \Big| \E{  h(F_y) h(F_z)  } \Big|.
\end{align*}
If $z\in J$, $y\notin J$ and $\|z-y\|=r$, then $m-r < \|z\| \le m$. The number of such $z$ is bounded above by $(2m+1)^d - (2(m-r)+1)^d$, which is at most $2dr (2m+1)^{d-1}$. So, for a given $r$ the inner sum in the first term is $\cO[2dr (2m+1)^{d-1} (2r+1)^d \exp( - c r)]$. Moreover, for a given $r>m$, the inner sum in the second term is of order $(2m+1)^d (2r+1)^d \exp( - c r)$. This shows $| \wt \sigma^2  - Var(\sum_{z\in J} h(F_z) ) / \# J | \le  \# J^{1/d} $. Using the stationarity, we can apply this estimate to the second term in \eqref{E:NormalApproximation2}.

We can apply Lemma~\ref{L:Covariance} to see
\begin{align}
\begin{split}\label{E:NormalApproximation3}
				&\Biggl|		\E{ \exp\Big( it \sum_{i} \xi_i \Big) - \exp\Big( it \sum_{i} \xi'_i \Big) } \Biggl| \\
				&\le \Biggl | \cov{ \exp\Big( it \sum_{i: i\neq j} \xi_i \Big)}{ \exp\Big( it \xi_{j} \Big) } \Biggl |\\
				&\quad + \Biggl | \E{ \exp\Big( it \sum_{i: i\neq j} \xi_i \Big) - \exp\Big( it \sum_{i: i\neq j} \xi'_i \Big) } \Biggl | \E{ \exp\Big( it \xi_{j} \Big) } \\
				&\le c_1 \exp\Big( -c_2 n^{1/2 - \epsilon} \Big) + \E{ \exp\Big( it \sum_{i: i\neq j} \xi_i \Big) - \exp\Big( it \sum_{i: i\neq j} \xi'_i \Big) }  \\
				&\le c_1 \ k^d \ \exp( -c_2 n^{1/2 - \epsilon} ),
\end{split}\end{align}
Finally, by the Esseen's lemma (\cite{petrov2012sums}, p. 109, Lemma 1), 
\begin{align}
\begin{split}\label{E:NormalApproximation4}
		& \Big| e^{ -t^2/2 } - \EE{ \exp\big( it \sum_i \xi'_i / ( k^d \E{ \xi_i^2 } )^{1/2} \big) } \Big| \\
		& \le 16 \ \frac{ \E{|\xi_i |^3} }{\E{\xi_i^2}^{3/2}} k^{-d/2} |t|^3 e^{-t^2/2} \le C k^{-d/2} |t|^3 e^{-t^2/2} ,
\end{split}
\end{align}
whenever, $|t| \le k^{d/2} / 4 \ \cdot \ \EE{\xi_i^2}^{3/2} / \EE{|\xi_i|^3}$. The last inequality in \eqref{E:NormalApproximation4} follows because $\EE{| \sum_{z\in J} h(F_z ) |^4 } \le C (\# J)^2$ for a constant $C$ which does not depend on $J$; this can be seen with similar stabilization techniques as in the proof of Lemma~\ref{L:Covariance}, we skip the details.

Using the estimates from \eqref{E:NormalApproximation1} to \eqref{E:NormalApproximation4} in combination with Esseen's theorem, yields two constants $K_1,K_2\in\R_+$, independent of $T$, such that
\begin{align*}
		| \p( Z_n \le z) - \Phi(z) | &\le K_1 \int_{-T}^{T} \frac{tn^{-\epsilon/2}}{t} \intd{t} + K_1 \int_{-T}^{T} \frac{k^d  e^{-c_2 n^{1/2 - \epsilon} } }{t} \wedge 1 \ \intd{t} \\
		&\quad + K_1 \int_{-T}^{T} \frac{tn^{-\epsilon}}{t} \intd{t} + K_1 \int_{-T}^{T} k^{-d/2} |t|^2 e^{-t^2/2} \ \intd{t} + \frac{K_2}{T} \le C n^{-\epsilon/4},
\end{align*}
for a choice of $T$ proportionally to $n^{\epsilon/4}$. This completes the proof.
\end{proof}

\begin{proof}[Proof of Proposition~\ref{P:MaximalInequality}]
Set $\phi(n) \coloneqq \sqrt{2 \wt \sigma^2 (2n+1)^d \log\log n}$, $S_n = \sum_{z: \|z\|\le n} h(F_z)$ and $r \coloneqq \floor{n^{1/6}}$, $k\coloneqq  \floor{n/k}$. Also set
\[
		E_j\coloneqq \{ |S_i| < \beta \phi(n), \> \forall i < j \} \cap \{ |S_j| \ge \beta \phi(n) \}.
\]
Then we obtain with the abbreviation $\xi_i = \sum_{z: \|z\| = i } h(F_z)$
\begin{align}
			\p( \max_{1\le j \le n} |S_j| \ge \beta \phi(n) ) &\le \sum_{i=0}^{k-2} \p\Big( \bigcup_{j=1}^r \Big( E_{ir+j} \cap |S_n-S_{ir+j}| \ge \epsilon \phi(n) \Big) \Big) \nonumber \\
			&\quad + \p\Big( \bigcup_{\ell=(k-1)r+1}^n \Big( E_{\ell} \cap |S_n-S_{\ell}| \ge \epsilon \phi(n) \Big) \Big) \nonumber \\
			&\quad + \p(|S_n| \ge \beta(1-\epsilon) \phi(n)) \nonumber\\
			&\le \sum_{i=0}^{k-2} \p\Big( \bigcup_{j=1}^r \Big( E_{ir+j} \cap |S_n-S_{(i+2)r}| \ge \frac{\epsilon}{2} \phi(n) \Big) \Big) \label{E:MaximalInequality1} \\
			&\quad + \sum_{i=0}^{k-2}  \p\Big( |\xi_{ir+1}| + \ldots + |\xi_{ir+2r}| \ge \frac{\epsilon}{2} \phi(n) \Big) \label{E:MaximalInequality2} \\
			&\quad + \p\Big( |\xi_{(k-1)r+1}| + \ldots + |\xi_{n}| \ge \epsilon \phi(n) \Big) \label{E:MaximalInequality3} \\
			&\quad + \p(|S_n| \ge \beta(1-\epsilon) \phi(n)) . \nonumber 
\end{align}
Consider the term in \eqref{E:MaximalInequality1}. Each single summand is
\begin{align}
		&\p\Big(	\bigcup_{j=1}^r \Big\{ |S_{ir+j}|\ge \beta\phi(n), |S_u| < \beta \phi(n), \quad \forall u<ir+j \Big\} \cap \Big\{ |S_n -S_{(i+2)r } |\ge \frac{\epsilon}{2} \phi(n)	\Big\} \Big) \nonumber \\
		\begin{split}\label{E:MaximalInequality4}
		&\le \p\Big(	\bigcup_{j=1}^r \Big\{ |S_{ir+j}|\ge \beta\phi(n), |S_u| < \beta \phi(n), \quad \forall u<ir+j \Big\} \Big)\\
		&\quad \cdot \p(  |S_n -S_{(i+2)r } |\ge \phi(n) \ \epsilon / 2  ) + c_1 (2n+1)^d  e^{-c_2 r },
		\end{split}
\end{align}
where we have used the following result: Let $I,J \subset \Z^d$ be two disjoint sets with $r = d(I,J)/2$. Then, using the stabilizing property of the functional $H$, we have
\begin{align*}
			&| \p( (F_z: z\in I) \in A , (F_z: z\in J) \in B ) \\
			&\qquad - \p( (F_z: z\in I) \in A ) \ \p( (F_z: z\in J) \in B ) | \le  c_1 (\# I + \# J ) e^{-c_2 r} ,
\end{align*}
where the constants $c_1,c_2 \in\R_+$ depend on $H$. This shows then \eqref{E:MaximalInequality4}.

The second factor in the first term in \eqref{E:MaximalInequality4} is ultimately smaller than 1/2. Indeed,
\begin{align*}
			&\p(  |S_n -S_{(i+2)r } |\ge  \phi(n) \ \epsilon / 2  ) \le \Big( \frac{\epsilon}{2} \phi(n)  \Big)^{-2} \E{ |S_n -S_{(i+2)r} |^2	} \\
			&= \Big( \frac{\epsilon}{2} \phi(n)  \Big)^{-2} \sum_{y,z: (i+2)r < \|y\|, \|z\|\le n } \E{ h(F_y) h(F_z) }  \le C / \log \log n.
\end{align*}
We continue with \eqref{E:MaximalInequality2}. Since the fourth moment of $h(F_z)$ exists, the Markov inequality yields that this term is of order
\begin{align*}
			&\sum_{i=0}^{k-2} (2r)^4 \epsilon^{-4} (2n+1)^{-2} (\log\log n )^{-2} \\
			&= (k-1) (2r)^4 \epsilon^{-4} (2n+1)^{-2} (\log\log n )^{-2} \le C \frac{r^4 (k-1)}{\epsilon^4 n^2 (\log\log n)^2 }.
\end{align*}
The same reasoning applies to \eqref{E:MaximalInequality3}, which is of order
\begin{align*}
			&\frac{ |n-(k-1)r|^4  }{\epsilon^{4} n^{2} (\log\log n)^{2} } \le C \frac{r^4 }{\epsilon^4 n^2 (\log\log n)^2 }.
\end{align*}
Collecting the estimates for the terms in \eqref{E:MaximalInequality1} to \eqref{E:MaximalInequality3} yields for $n$ sufficiently large
\begin{align}\begin{split}\label{E:MaximalInequality5}
				\p( \max_{1\le j \le n} |S_j| \ge \beta \phi(n) ) &\le C (2n+1)^d c_1 e^{-c_2 r} + C \frac{r^4 k}{\epsilon^4 n^2 (\log\log n)^2 } \\
				&\quad + \frac{1}{2}\p( \max_{1\le j \le n} |S_j| \ge \beta \phi(n) )  + \p(|S_n| \ge \beta(1-\epsilon) \phi(n)) .
\end{split}\end{align}
Rearranging the terms in \eqref{E:MaximalInequality5} and using the normal-approximation estimate from Proposition~\ref{P:NormalApproximation}, we obtain
\begin{align*}
		\p( \max_{1\le j \le n} |S_j| \ge \beta \phi(n) ) &\le C \frac{r^3}{\epsilon^4 n} + 2 \p( |S_n| \ge \beta (1-\epsilon) \phi(n)) \le C (\log n)^{-(1+\rho) }
\end{align*}
for some $\rho\in\R_+$. This completes the proof.
\end{proof}

\subsection{Results for the invariance principle}

\begin{lemma}\label{L:BMContinuity}
Let $B$ be a standard Brownian motion. Let $(t_n)_n$ be of order $o(n)$. Then
\begin{align}\label{E:BMContinuity1}
	\limsup_{n\to\infty}		\max_{k: \ k\le n} \, \sup_{x: \ |x-k\sigma^2| \le t_n} \,  \frac{ |B(x)-B(k\sigma^2)| }{ \sqrt{ 2 t_n \log( n\sigma^2 t_n^{-1} ) }} \le 1  \quad a.s.
\end{align}
\end{lemma}
\begin{proof}
We use the fact that $t\mapsto a^{-1/2} B(a t)$ is a time-changed Brownian motion for $a>0$ and make the definition $W(x) = B(x (n\sigma^2) )(n\sigma^2)^{-1/2}$, $x\ge0$. Then the maximum on the left-hand side equals
\[
	\max_{k: \ k\le n} \,  \sup_{x: \  |x(n\sigma^2)^{-1}-kn^{-1} | \le t_n (n\sigma^2)^{-1} } \,  |W(x(n\sigma^2)^{-1})-W(kn^{-1})| \ (n\sigma^2)^{1/2}. 
\]
Plainly, this is at most $
	\sup_{x,y: \ |x-y|\le t_n (n\sigma^2)^{-1} } |W(x)-W(y)| \ (n\sigma^2)^{1/2}$.
 L{\' e}vy's modulus of continuity theorem (in the global version) yields
 \[
 	\limsup_{\epsilon\downarrow 0} \sup_{ \substack{ x,y \in [0,1],\\ |x-y| \le \epsilon}} \frac{|W(x)-W(y)|}{\sqrt{2\epsilon \log \epsilon^{-1} } } = 1 \quad a.s.
	\]
Consequently, the left-hand side of \eqref{E:BMContinuity1} is bounded above by
1 $a.s.$
\end{proof}

\end{document}